\pdfoutput=1
\documentclass{amsart}

\usepackage{mathtools}
\usepackage{amsmath}
\usepackage{amsfonts}
\usepackage{amssymb}
\usepackage{graphicx}
\usepackage{graphicx, xcolor}              
\usepackage{amsmath}               
\usepackage{amsfonts}              
\usepackage{amsmath}
\usepackage{amsthm}  
 \usepackage{float} 
 \usepackage{MnSymbol}
 \usepackage{tikz}
 \usepackage{enumerate}
\usetikzlibrary{shapes.geometric}
\usetikzlibrary{arrows.meta,arrows}

\usepackage{caption}
\usepackage{wrapfig}
\usepackage{subcaption}
\usepackage{color}
\usepackage{hyperref}

\usepackage{mathrsfs}
\usepackage{url}
\usepackage{cite}

\newtheorem{note}{Note}
\newtheorem{theorem}{Theorem}

\theoremstyle{plain}

\newtheorem{corollary}{Corollary}

\newtheorem{definition}{Definition}

\newtheorem{lemma}{Lemma}

\newtheorem{proposition}{Proposition}

\numberwithin{equation}{section}

\newtheorem{remarks}{Remarks}

\begin{document}

\title{ The Clock Theorem for knotoids and linkoids}

\author{Neslihan G{\"u}g{\"u}mc{\"u}}
\author{Louis H.Kauffman}

\address{Neslihan G{\"u}g{\"u}mc{\"u}:  Department of Mathematics, Izmir Institute of Technology, G\"ulbah\c ce, Urla Izmir 35430, Turkey}
\address{Louis H.Kauffman:Department of Mathematics, Statistics and Computer
Science, University of Illinois at Chicago, 851 South Morgan St., Chicago
IL 60607-7045, U.S.A. and
}
\email{neslihangugumcu@iyte.edu.tr} \email{kauffman@math.uic.edu; loukau@gmail.com}

\maketitle
\begin{abstract}
In this paper, we generalize the \textit{Clock Theorem} of Formal Knot Theory to knotoids in $S^2$. The clock theorem implies that clock states of a knotoid diagram form a lattice under transpositions.  These states form the basis of many invariants of knotoids and linkoids including the Alexander polynomial, Mock Alexander polynomial and the Jones polynomial. \end{abstract}

\section{Introduction}

In this paper we generalize the Clock Theorem in Formal Knot Theory \cite{FKT} from classical knot and link diagrams to knotoids. This generalization applies to the structure of Mock Alexander polynomials  \cite{MAP1, MAP2} and their generalizations. In this work, the knotoid polynomials are obtained by state summations (the clock states)  where each state of the diagram contributes a monomial to a sum over states that is equal to the full polynomial. Figure \ref{fig:intro} illustrates all clock states of a knotoid diagram $K$. Note from Figure \ref{fig:intro} that each state has one marker in each of the unstarred regions of the diagram where the unstarred regions are in one-to-one correspondence with the crossings of the diagram.

The contribution of a clock state to the Mock Alexander polynomial is the product of the weights assigned on the local regions where state markers coincide with the placements of the vertex weights. In this example we  see the contribution of each clock state and find that the Mock Alexander polynomial of $K$,  $\nabla_{K} (W) = W^2 - W^{-1} + W$.  For more information about this method, the reader can consult the book and papers mentioned above.
   \begin{figure}[H]
\centering
\includegraphics[width=.75\textwidth]{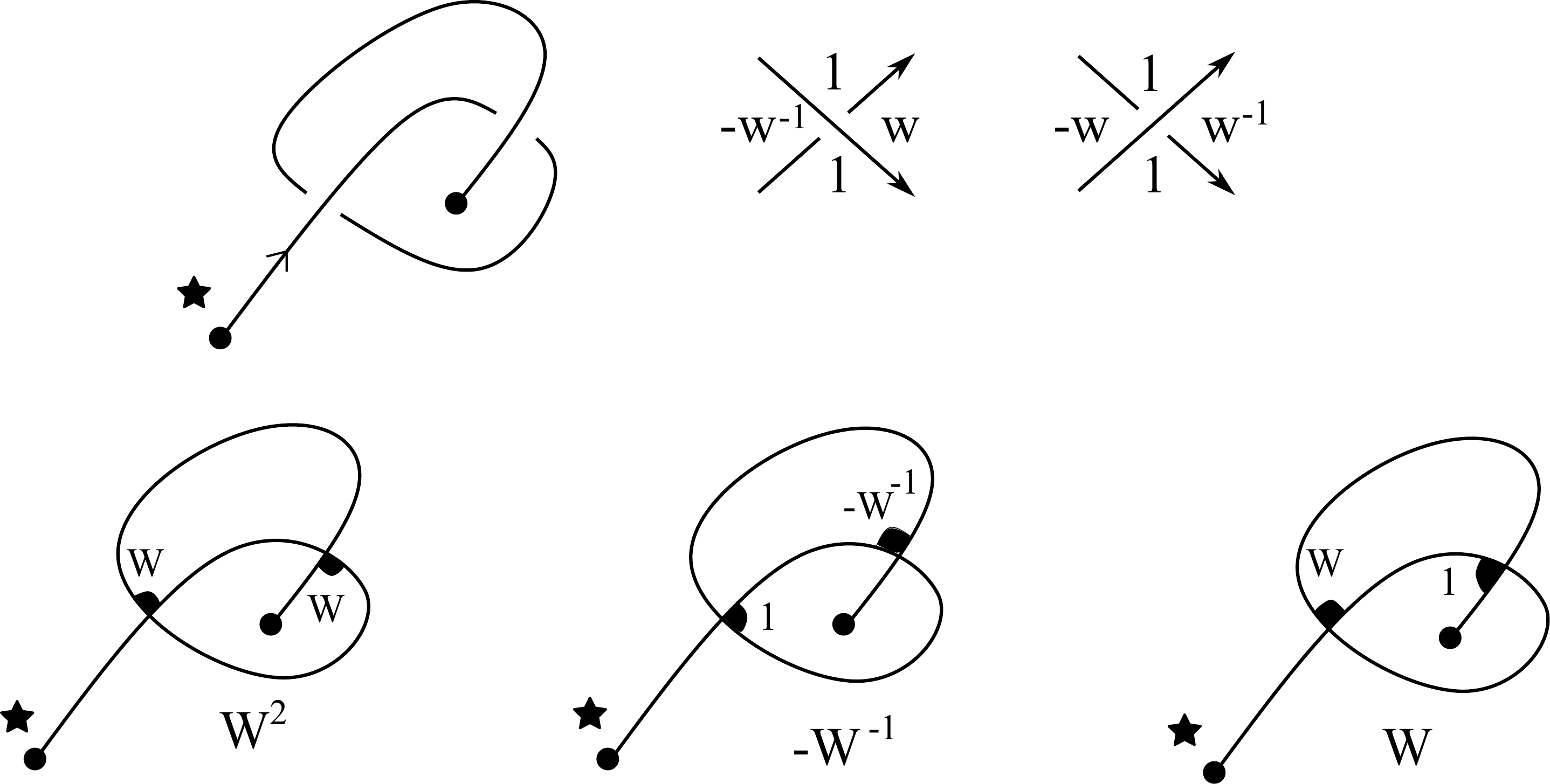}
\caption{Clock states of a starred knotoid diagram and their contributions.}
\label{fig:intro}
\end{figure}

Originally the Clock Theorem enabled a proof that the state summation model given in \cite{FKT} yields the Alexander-Conway polynomial. Our generalization of the Clock Theorem to knotoids has a primarily structural relationship to the Mock Alexander polynomial. The generalized Clock Theorem is not needed as a logical underpinning for the development of the Mock Alexander polynomials, but it can be used to prove results about them.
Our primary goal in this paper is prove the generalized Clock Theorem and we give applications of it at the end of the paper.

The generalized Clock Theorem shows that the states underlying the Mock Alexander polynomial have a lattice structure and that any two states are related by local moves that we call \textit{clock moves}. Each local move has a rotational sense in the plane and so are designated as clockwise and counter-clockwise moves. See Figure \ref{fig:clockmovesintro}.

   \begin{figure}[H]
\centering
\includegraphics[width=.75\textwidth]{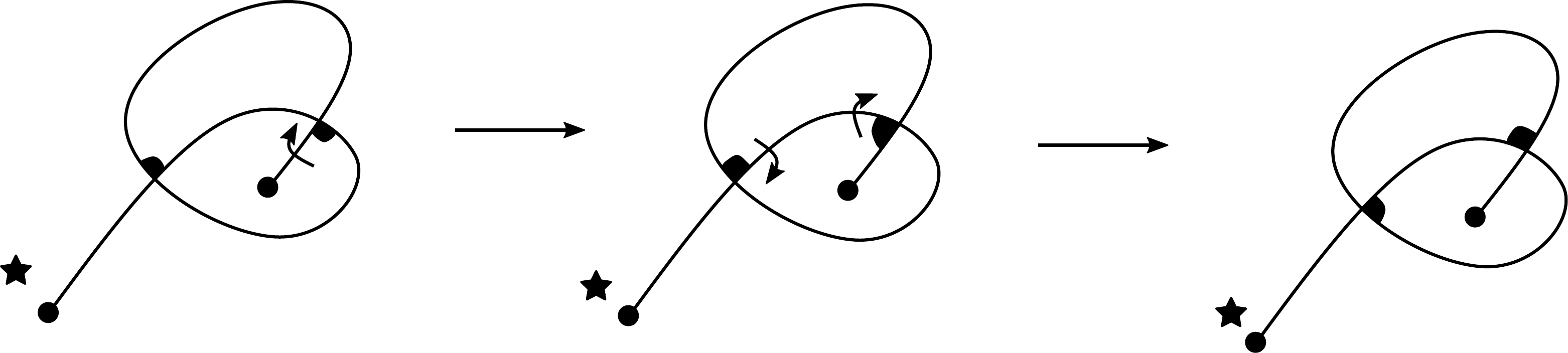}
\caption{Clock states of $K$ related to each other by clock moves.}
\label{fig:clockmovesintro}
\end{figure}

 A state that has only clockwise moves is said to be a \textit{clocked state} and a state that has only counterclockwise moves is said to be \textit{counter-clocked}. In this paper  we prove the Clock Theorem for knotoids: The clocked and counter-clocked states for a knotoid diagram exist uniquely and that they constitute the top and the bottom of the lattice of states. Furthermore, we prove that the states of a knotoid diagram are in one-to-one correspondence with the set of Euler-Jordan trails from the initial point of the knotoid to its endpoint. These trails correspond to walks on the knotoid diagram that comprise all the edges of the shadow graph of the knotoid and such that the walk never crosses at a crossing, instead turning left or right at the intersection when it meets the crossing in the shadow diagram.

The purpose of this paper is to give a careful proof of the Clock Theorem for 1-linkoids. We will write a sequel to this paper with applications of the Theorem. 

The paper is organized as follows: Section 2 describes the Mock Alexander polynomial and how it can be computed in terms of clock states and in terms of trails. The section then details the definitions of clock moves, shelling algorithms, trails and the needed lemmas to construct a proof of the Clock Theorem. Section 3 discusses applications of the Clock Theorem for knotoids and plans for further research.


\section{A generalized Alexander polynomial for linkoids}

The theory of knotoids was introduced by Turaev \cite{Turaev}. Knotoids and their variants such as linkoids, starred knotoids, have been studied further in \cite{GK1, MAP1,MAP2, GG, BKP, A}. In this section we make a quick review of linkoids and starred linkoids, and a generalization of the Alexander-Conway polynomial for starred linkoids that we call the Mock Alexander polynomial.

\begin{definition}\normalfont
An \textit{$n$-linkoid diagram} in a surface is an immersion of a number of unit circles and exactly $n \geq 0$ unit intervals into the surface. The images of circles are \textit{knot components}, the images of unit intervals are  \textit{knotoid components} of a linkoid diagram. Specifically, if a linkoid diagram consists of only one knotoid component then it is called a \textit{knotoid} diagram, and a linkoid diagram with only a number of knot components is a link diagram. The images of $0$ and $1$ are considered to be distinct endpoints and called the \textit{tail} and the \textit{head} of a knotoid component, respectively.

Linkoid diagrams are considered up to the equivalence relation induced by Reidemeister moves that take place away from endpoints.  Indeed, it is forbidden to pull/push an arc with an endpoint. Also, each of the components of a linkoid diagram admits an orientation such that the knotoid component is oriented from the tail to the head.


\end{definition}

\begin{definition}\normalfont
A 1-linkoid diagram with endpoints that lie in the same region of the surface is called a \textit{knot-type} 1-linkoid diagram. Otherwise, a 1-linkoid diagram is called a \textit{proper} 1-linkoid diagram.

\end{definition}

We will restrict our attention to 1-linkoids in this paper. We define below a variant of a 1-linkoid diagram that is endowed with a special decoration. We call these variants \textit{starred linkoids}. Starred linkoids were studied in \cite{MAP1, MAP2, FKT} in the constructions of a generalized Alexander-Conway polynomial for linkoids.

\begin{definition} \label{def:starred}\normalfont
A \textit{starred 1-linkoid diagram} is a 1-linkoid diagram one of whose regions is endowed with a star.
\end{definition}

For any 1-linkoid diagram in $S^2$ , we can choose one of the regions containing an endpoint of the 1-linkoid and place a star in that region. With this choice, the starred region can be depicted as the exterior (unbounded) region for the 1-linkoid, represented in the plane. In such a representation, a point (call it $\infty$) is removed from the two-sphere. Call this starred planar diagram a {\it standard representation} of the 1-linkoid. We will consider standard representations of 1-linkoid diagrams in the sequel.

Starred 1-linkoid diagrams are considered up to the equivalence relation induced by Reidemeister moves which avoids the starred region. 


\begin{definition}\normalfont
The \textit{universe} of a linkoid diagram (starred or unstarred) is the graph that is obtained by ignoring the under/over information of crossings of the linkoid diagram.

\end{definition}

\begin{definition}\normalfont
 Let  $L$ be a connected starred 1-linkoid diagram.  A \textit{clock state} of $L$ is obtained by placing a marker at every bounded face of the universe of $L$,  placed at exactly one of the crossings that is incident to the region.
 
See the full list of clock states of the given knotoid diagram in Figure \ref{fig:clstates} where a marker at a crossing is given as a black marker.

\end{definition}

   \begin{figure}[H]
\centering
\includegraphics[width=.95\textwidth]{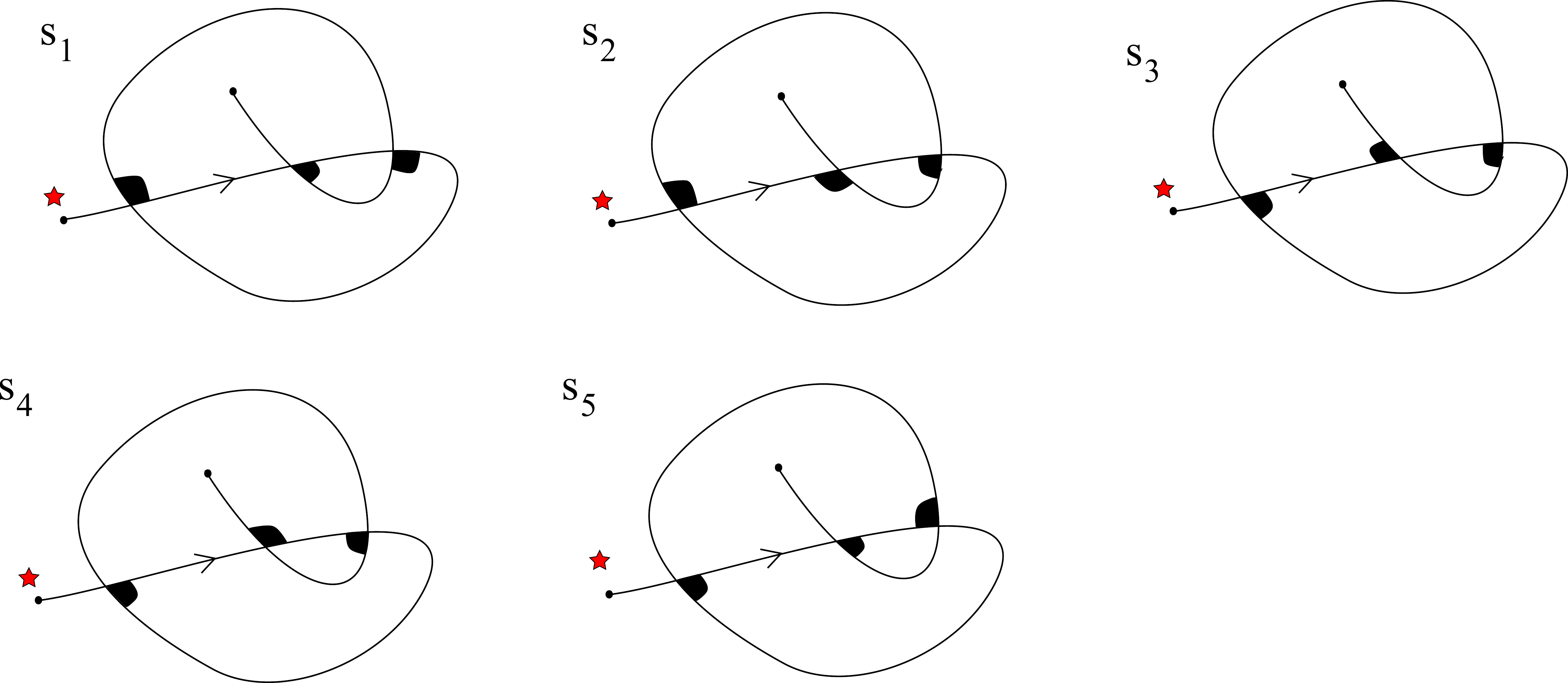}
\caption{All clock states of a starred knotoid diagram.}
\label{fig:clstates}
\end{figure}
Note that it easily follows from the Euler's formula that a starred 1-linkoid diagram admits equal number of unstarred regions and crossings. This enables that any starred 1-linkoid diagram admits a clock state.  

\begin{figure}[H]
\centering
\includegraphics[width=.35\textwidth]{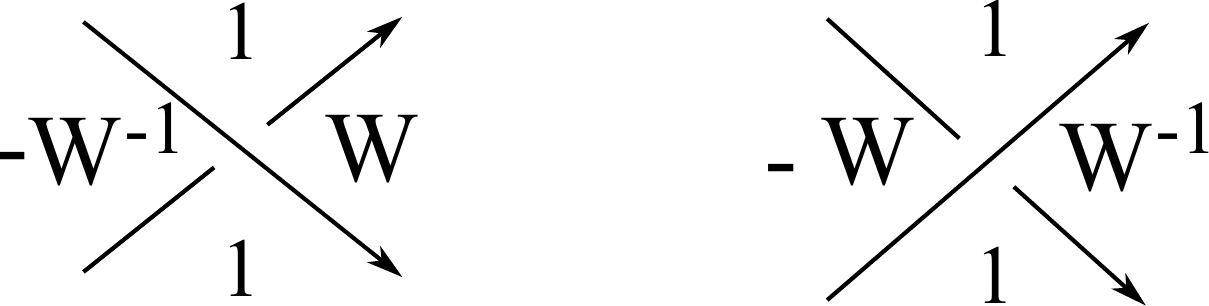}
\caption{Labels at local regions incident a positive and negative crossing.}
\label{fig:labelgeneral}
\end{figure}

To define the Mock Alexander polynomial for connected 1-linkoids in $S^2$, we first define the weights at crossings.  We assume any connected 1-linkoid diagram $L$ is oriented and its knotoid component is oriented from its tail to its head, and we assume that $L$ is endowed with a star in the region that is adjacent to its tail.

\begin{definition}\normalfont
Let $L$ be a connected, oriented and starred $1$-linkoid diagram.The \textit{weight} of a clock state $S$ of $L$, denoted by  $<L~ |~ S>$  is defined to be the product of local weight labels that the markers of the state $S$ indicate at crossings. See Figure \ref{fig:labelgeneral} for  local weights assigned at a positive and a negative crossing.
 \end{definition}

\begin{definition}\normalfont \cite{MAP1}
Let $L$ be an oriented starred 1-linkoid diagram in $S^2$. The \textit{Mock Alexander polynomial} of $L$ is the Laurent polynomial with integer coefficients defined as,

$$\nabla_{L} (W) = \sum_{S \in \mathcal{S} } < L ~|~ S >,$$
where $\mathcal{S}$ denotes the set of all clock states of $L$. 
\end{definition}

\begin{theorem}\cite{MAP1}
The Mock Alexander polynomial is an invariant of oriented, starred 1-linkoids in $S^2$.
\end{theorem}

\begin{remarks}\normalfont
\begin{enumerate}
\item 
As shown in \cite{FKT},  The Alexander-Conway polynomial of any oriented, connected link diagram $L$ is equal the Mock Alexander polynomial of a starred link diagram obtained by endowing any pair of adjacent regions of $L$ with stars.
\item The Mock Alexander polynomial can be defined for any admissable $n$-linkoid lying in a surface of genus $g$ for $g \geq 0$. See \cite{MAP1} for details. 

\end{enumerate}
\end{remarks}

\subsection{Clock states and Trails} \label{sec:trails}
Similar to the case of classical knots studied in \cite{FKT}, there is a 1-1 correspondence between clock states of a 1-linkoid diagram and what we call {\it trails} of the diagram. In this section we study this correspondence.

\begin{definition}\normalfont
A \textit{smoothing} of a crossing in a 1- linkoid diagram consists in removing the crossing and then connecting the resulting ends of the strands without creating a crossing. We show two possible smoothing of a crossing in Figure \ref{fig:smooth}.  The smoothing that removes the crossing in the vertical direction (with respect to the top to bottom direction of the plane) is called a \textit{vertical} smoothing, otherwise it is called a \textit{horizontal} smoothing.  
\end{definition}

\begin{figure}[H]
\centering
\includegraphics[width=.3\textwidth]{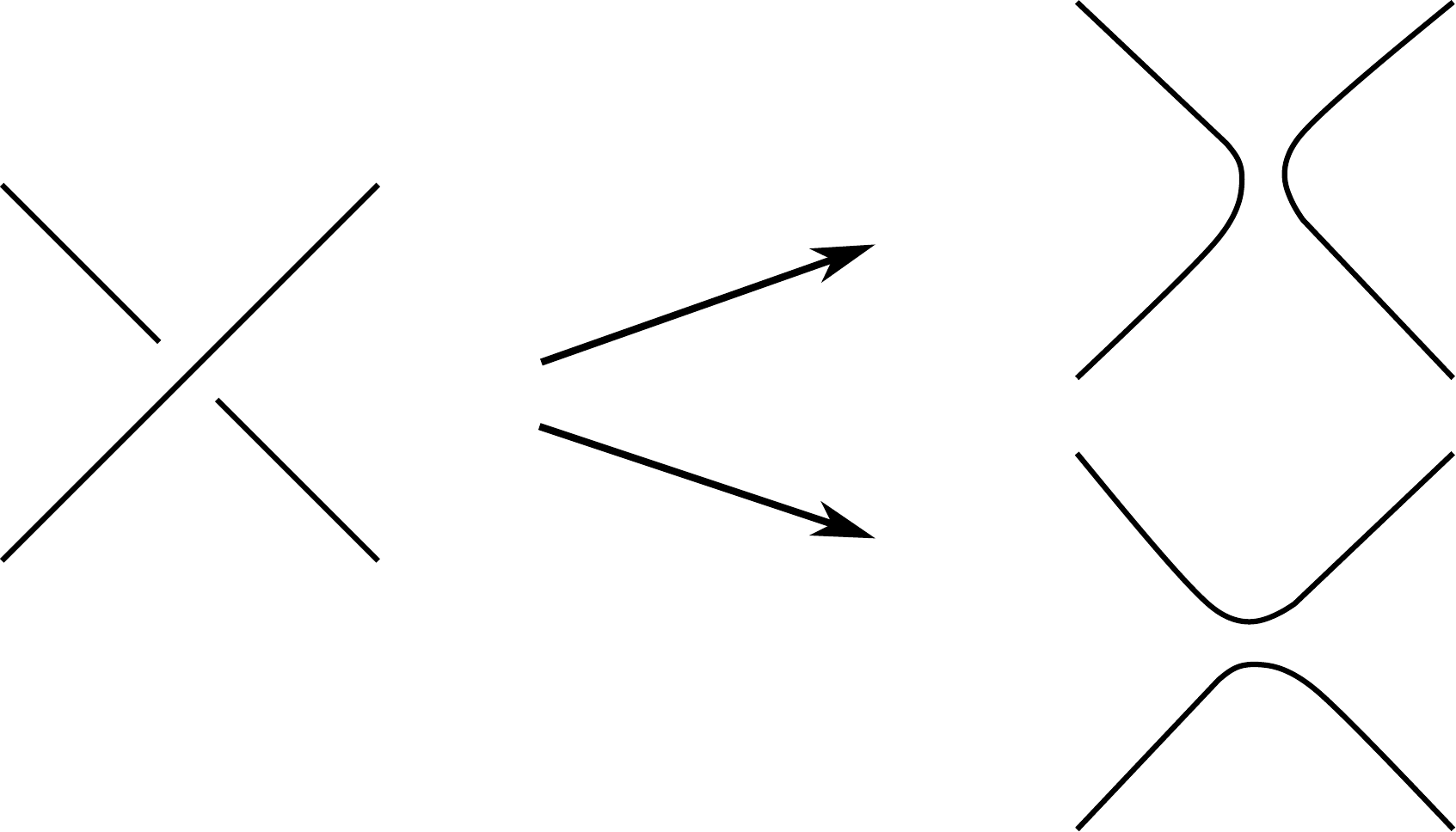}
\caption{The vertical and horizontal smoothings of a crossing.}
\label{fig:smooth}
\end{figure}


\begin{definition}\normalfont

A \textit{trail} of a 1-linkoid diagram is a simple (non-self intersecting) path that traverses each edge of the underlying graph of the diagram exactly once.   \end{definition}


\begin{proposition}
Every 1- linkoid diagram admits a trail.

\end{proposition}

\begin{proof}
We smooth each crossing of a 1-linkoid diagram in such a way that the connectivity is maintained at the end of the smoothing process. Such smoothing exists for a crossing. In fact, if one of the smoothings of a crossing disconnects the diagram, the other  smoothing keeps the diagram connected as the reader can verify easily.

\end{proof}

In Figure \ref{fig:trail} we list all trails of the knotoid diagram depicted in Figure \ref{fig:clstates}.
 
\begin{figure}[H]
\centering
\includegraphics[width=.93\textwidth]{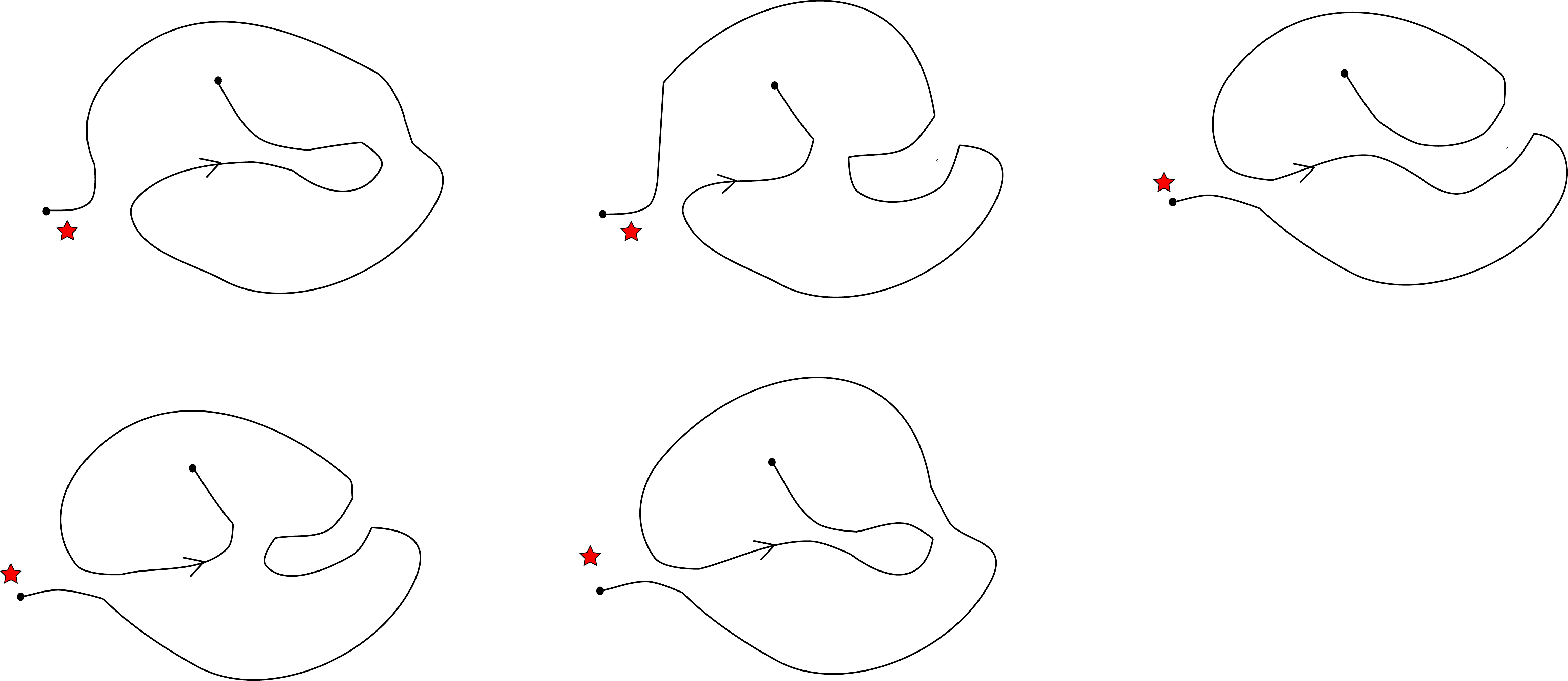}
\caption{Trails of a knotoid diagram.}
\label{fig:trail}
\end{figure}

\begin{proposition}
Let $L$ be a starred 1-linkoid diagram in $S^2$ such that the starred region is the exterior region. Every trail of $L$ determines a unique clock state of $K$.
\end{proposition}

\begin{proof}
Let $t$ be a trail of $L$. The trail $t$ determines a directed rooted tree underlying $K$ as follows. The root vertex is determined to be the star placed at the exterior region and each of the regions of $L$ other than the exterior region is endowed with a vertex. Each pair of vertices that lie in two distinct regions of $L$ is connected with an edge that passes through a smoothed site in $t$ to which both vertices are adjacent. The resulting graph is clearly a rooted tree. We endow each edge of the rooted tree with an arrow that is directed  towards the root of the tree. Let $\bf{t}$ denote the rooted tree that is induced by $t$ in this way.

We can consider the arrow on an edge of $\bf{t}$ as an indicator of a state marker at the crossing that the edge is passing through when it gets smoothed: A state marker is placed on the local region that is adjacent to the crossing and where the arrow goes in. In this way, we obtain a unique clock state of $L$ associated to the given trail $t$. 

See Figures \ref{fig:tree}  and \ref{fig:bttree} for an illustration.
\end{proof}

\begin{figure}[H]
\centering
\includegraphics[width=.75\textwidth]{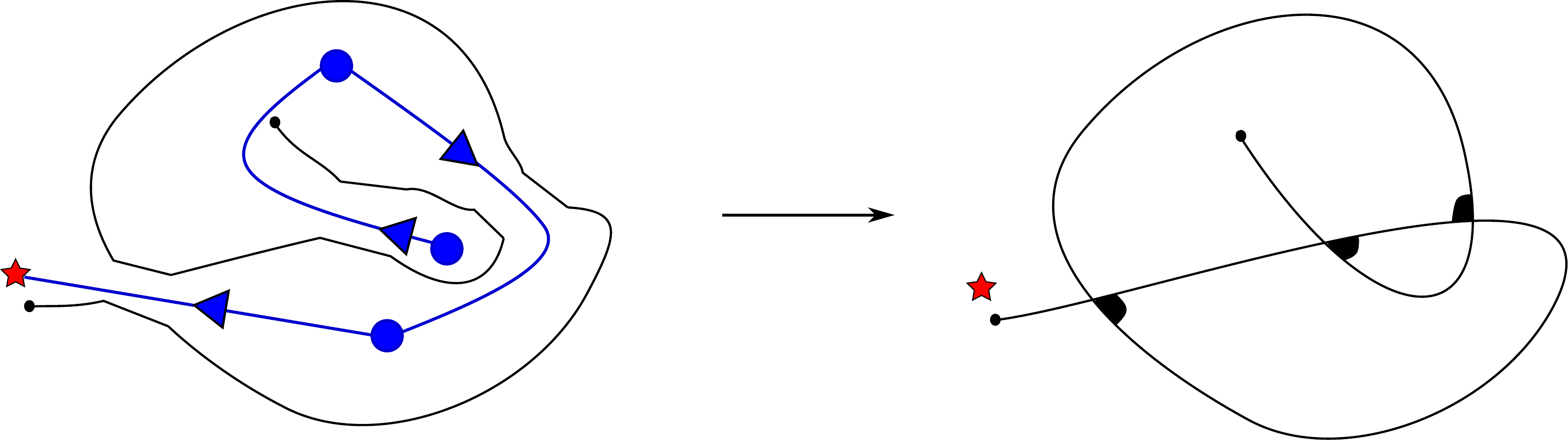}

\caption{The directed tree determined by a trail and the corresponding state.}
\label{fig:tree}
\end{figure}

\begin{proposition}\label{prop:Jordan}
Let $L$ be a starred 1- linkoid diagram. Each clock state of $L$ determines a  unique trail.
\end{proposition}

\begin{proof}
Let $s$ be a state of $L$. Every bounded region of $L$ receives exactly one state marker at an adjacent crossing. We smooth each crossing in the direction of the state marker.  We observe that each smoothing connects two distinct regions of $L$. To see this, assume to the contrary: Suppose a smoothing of a crossing with respect to the state marker direction connects a region with itself.  In this case, the smoothing disconnects $L$ into two components, and one of the components fails to be admissable (there is no state marking for that part because the number of regions is not equal to the number of vertices) hence $L$ would be non-admissable.  This yields a contradiction. See Figure \ref{fig:Jordan} for an illustration.

\begin{figure}[H]
\centering
\includegraphics[width=.45\textwidth]{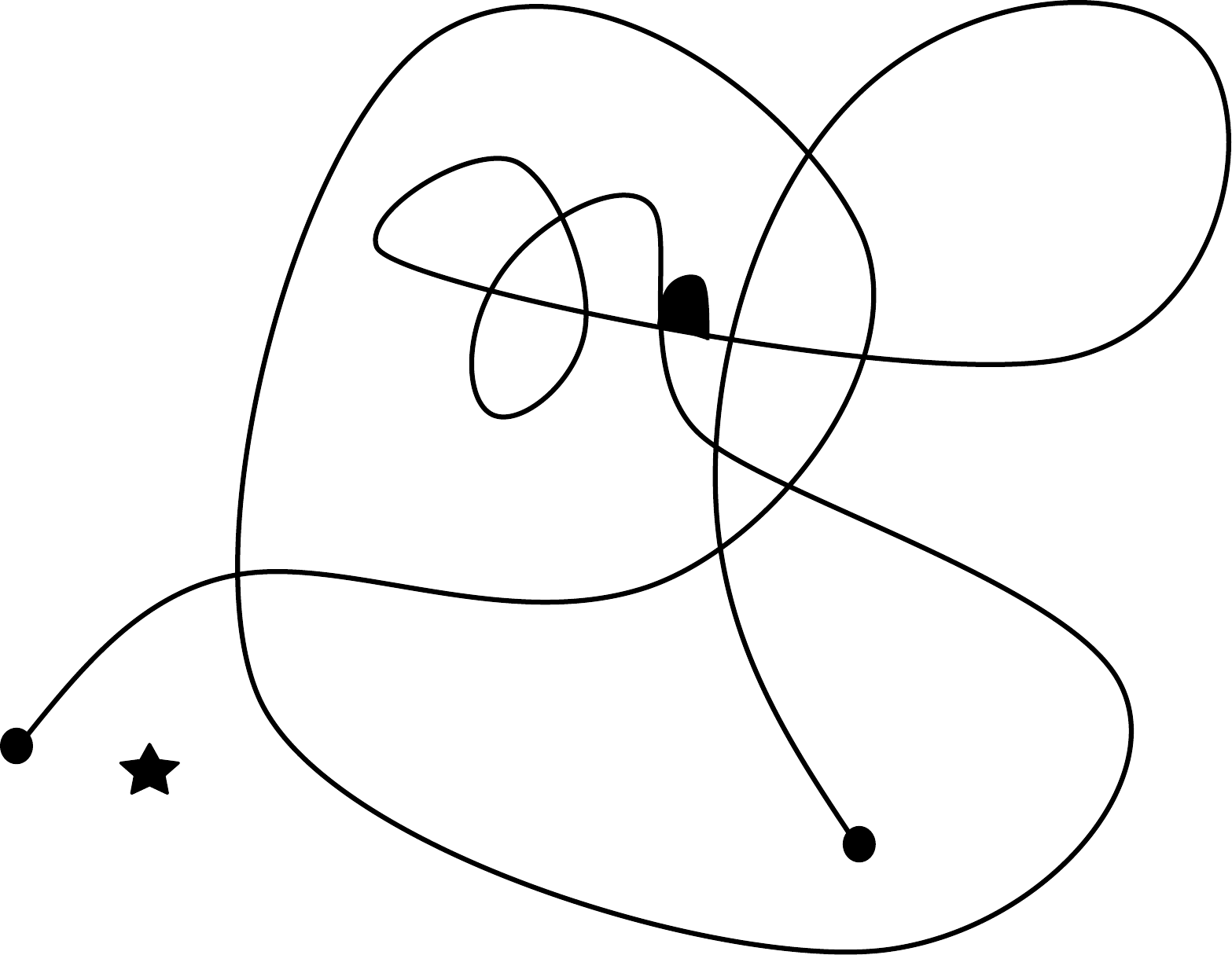}

\caption{A starred knotoid diagram and a state marker at crossing that disconnects the diagram by smoothing it.}
\label{fig:Jordan}
\end{figure}


Therefore smoothing each crossing of $L$ one by one connects each region of $L$ with the starred region. This implies that the resulting curve is a connected curve containing the two endpoints of $L$ and it admits only one region. Moreover the curve traverses each edge of $L$ exactly once. See Figure \ref{fig:trail1} for an illustration.

\end{proof}

\begin{figure}[H]
\centering
\includegraphics[width=.75\textwidth]{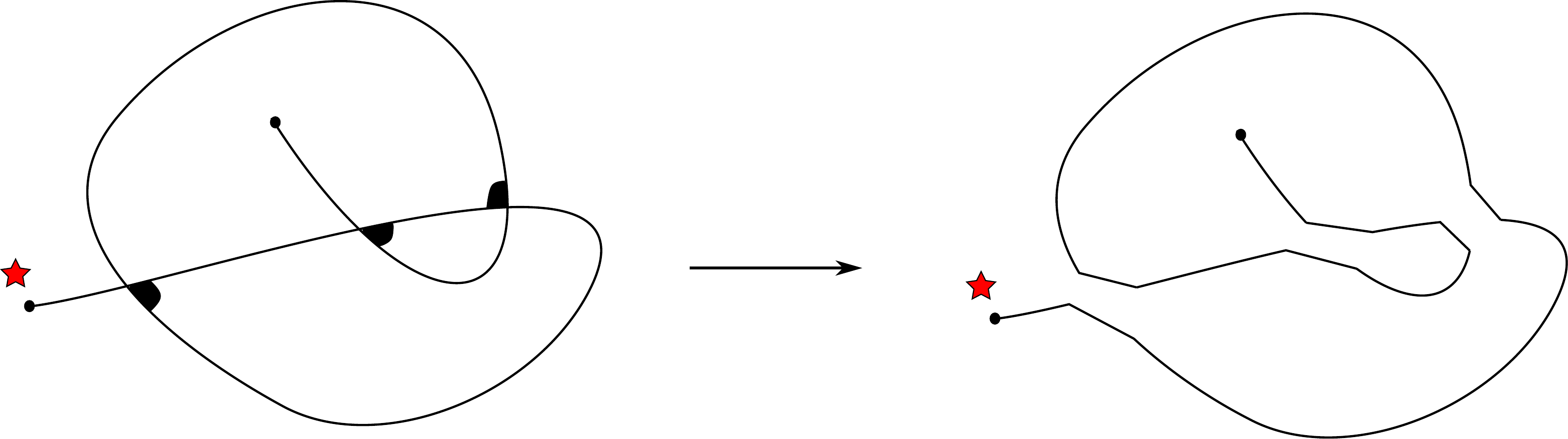}

\caption{The state $s_5$ and the corresponding trail}
\label{fig:trail1}
\end{figure}



\begin{corollary}\normalfont \label{cor:statesandtrails}
There is a one-to-one correspondence between clock states of a starred 1- linkoid diagram and its trails. 
\end{corollary}

Notice that the rooted tree that is induced by a trail of a starred 1-linkoid diagram $L$ comprises all vertices of the dual graph of  the universe of $L$. Then, we have the following.

\begin{corollary}
The number of clock states of a starred 1- linkoid $L$ is equal to the number of maximal trees of the dual graph of the universe of $L$. 
\end{corollary}

\begin{note}
The rooted tree that corresponds to a state of a starred 1-linkoid diagram may have branches. See Figure \ref{fig:bttree}.

\end{note}

\begin{figure}[H]
\centering
\includegraphics[width=.6\textwidth]{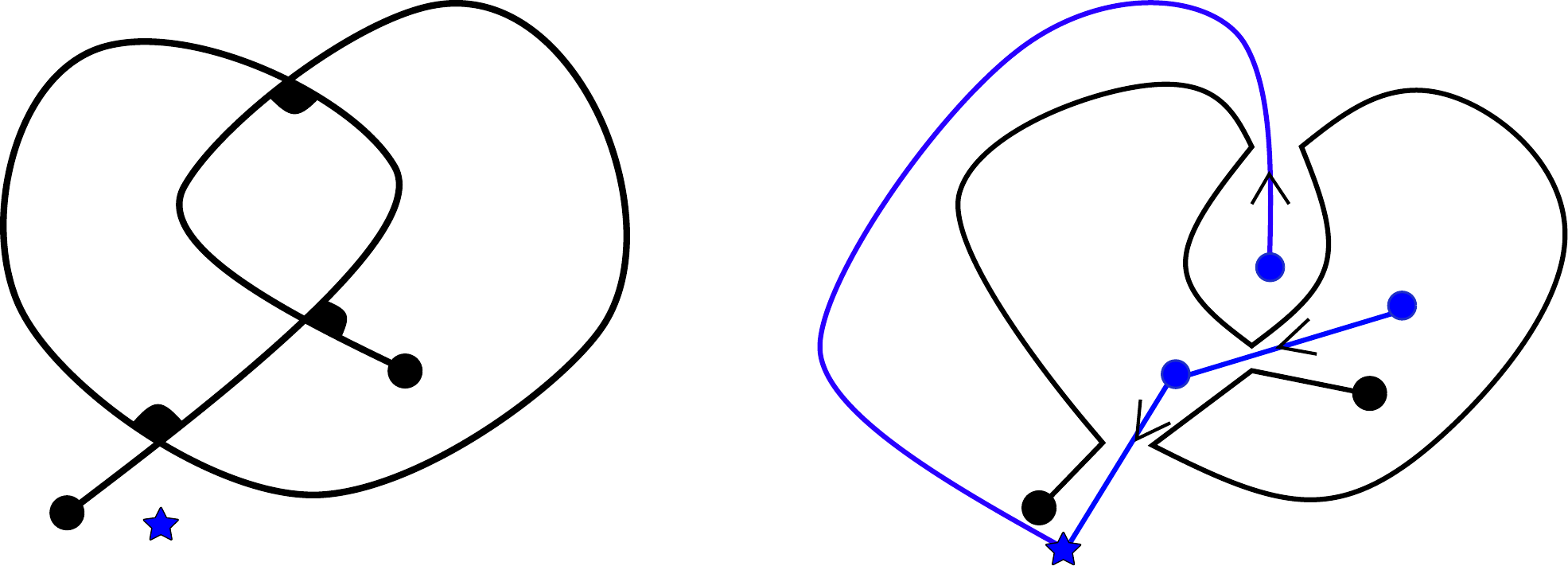}

\caption{A branched tree corresponding to a state.}
\label{fig:bttree}
\end{figure}




\subsection{Clock moves}

We begin this section by defining \textit{clock moves} on clock states of a starred 1-linkoid diagram $L$. Since every starred 1-linkoid diagram can be considered in the plane, we consider $L$ to lie in the plane.
\begin{definition}\normalfont

A \textit{clock move} on a state of a 1-linkoid knotoid diagram $L$ is the 90 degree rotation of either a pair of state markers that lie in regions of $L$ sharing a common boundary or of the state marker that is incident to the head of $L$. The result of a clock move on a clock state of $L$ is another clock state of $L$. A clock move can occur in either clockwise or counter-clockwise direction.

We illustrate clock moves in Figure \ref{fig:clock2} where  the first clock move swaps the markers that lie in distinct regions labeled $A$ and $B$ by a clockwise rotation, and the second clock move rotates the state marker that is incident to the head of the diagram in the clockwise direction, and moves the marker in the same region.  Note that the boundary edge that is shared by two regions may possibly contain a knotted portion. 
\end{definition}

\begin{figure}[H]
\centering
\includegraphics[width=1\textwidth]{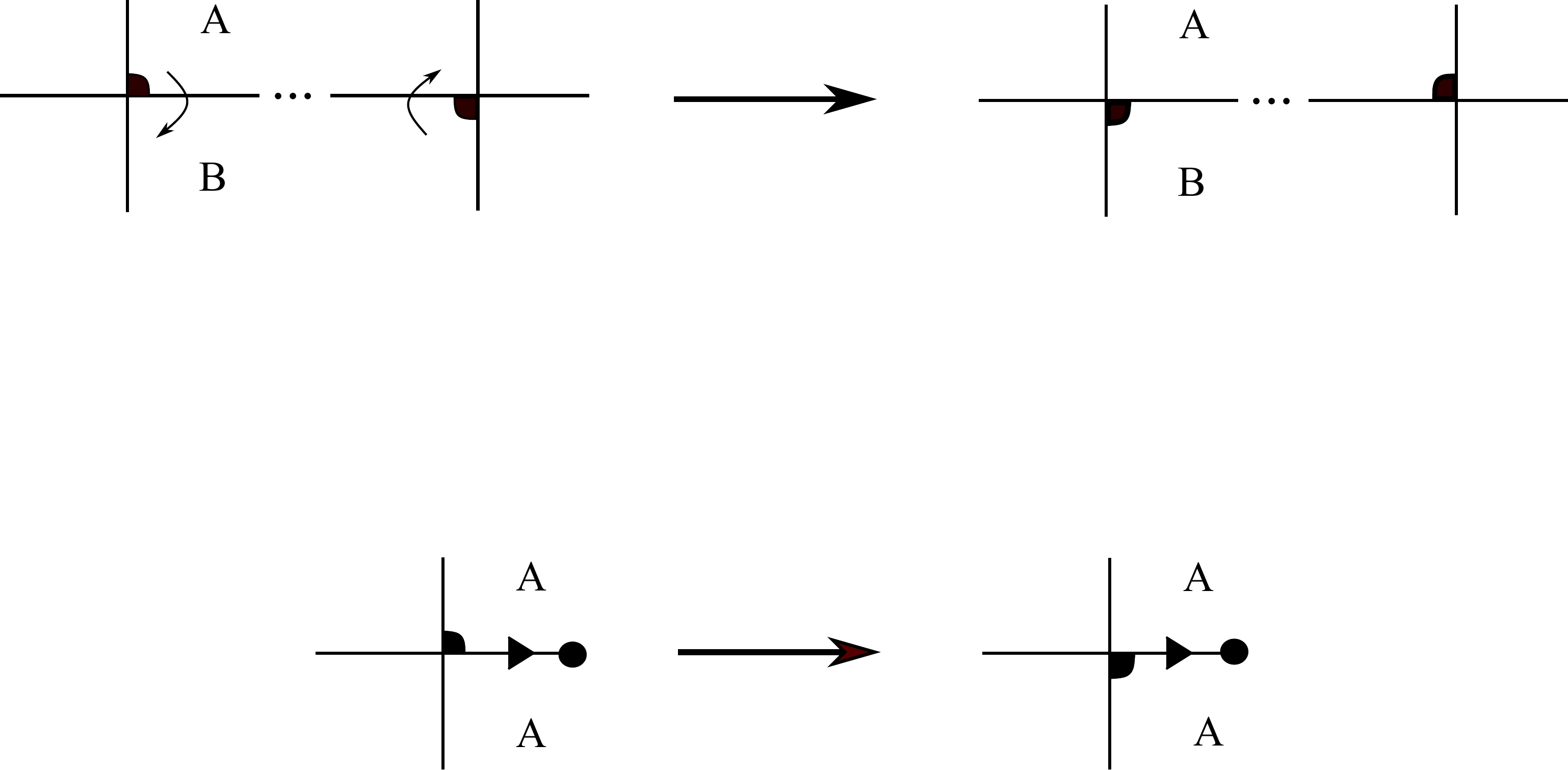}
\caption{Clock moves}
\label{fig:clock2}
\end{figure}

\begin{definition}\normalfont
A  \textit{clocked state} of a 1-linkoid diagram $L$ is a state that admits clock moves only in the clockwise direction.  A \textit{counter-clocked state} is a state that admits clock moves only in the counter-clockwise direction. 
A state of $L$ is called an \textit{extremal state} if it is either a clocked or a counter-clocked state. A state that admits clock moves both in the clockwise and counter-clockwise direction is called a \textit{mixed state}.

\end{definition}

\begin{definition}\normalfont
Let $T_1$, $T_2$ be two trails of $K$ that are determined from a clocked state and a counter-clock state of $L$, respectively. 
$T_1$ is called a \textit{clocked} trail and $T_2$ is called a \textit{counter-clocked} trail. If a trail  of $K$ is obtained from a mixed state, then it is called a \textit{mixed trail}.

\end{definition}

\begin{definition}\normalfont
Let $T$ be a trail of a 1-linkoid universe.  The \textit{resmoothing of a local site} of $T$ is to replace a local smoothing site that is a trivial 2-tangle with the opposite type of trivial 2-tangle,  as shown in Figure \ref{fig:reassem}.  A trail $T^{'}$ is said to be obtained from a trail $T$ by a \textit{single} or \textit{double exchange}  if $T^{'}$ is the result of reassembling one or two sites of $T$, respectively. We call a site where a single or double exchange takes place an \textit{exchange site}.

\end{definition}

For instance, the trails illustrated in Figure \ref{fig:site} are obtained from each other by a single exchange.
 \begin{figure}[H]
\centering
\includegraphics[width=.7\textwidth]{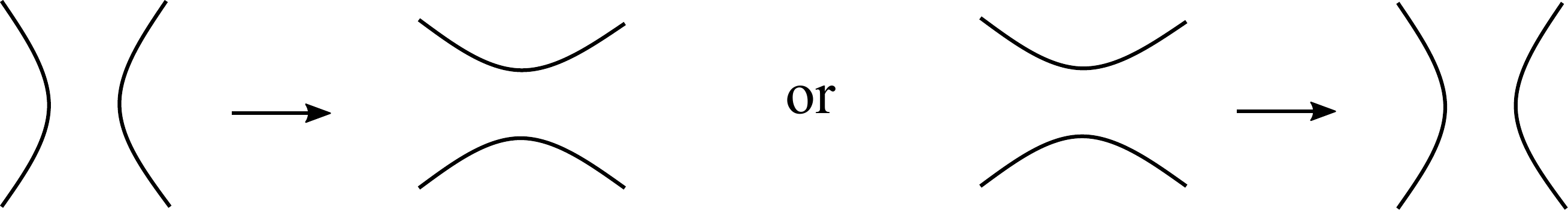}
\caption{The resmoothing of a site.}
\label{fig:reassem}
\end{figure}


\begin{figure}[H]
\centering
\includegraphics[width=.6\textwidth]{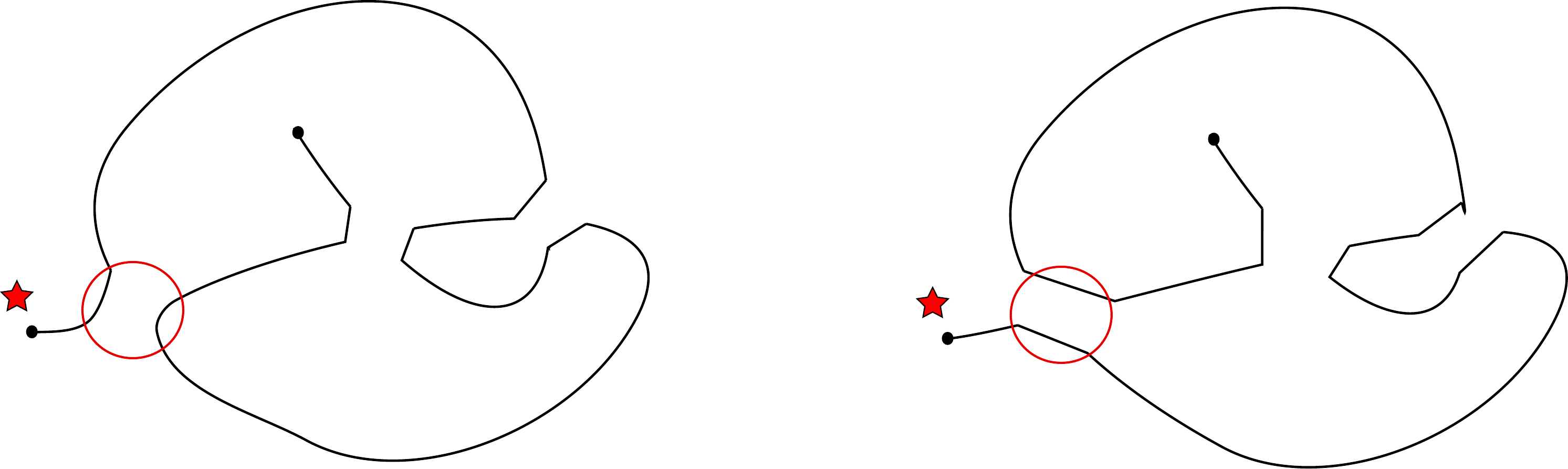}
\caption{The  trails corresponding to the states $s_2$ and $s_4$ of $K$ depicted in Figure \ref{fig:clstates}. They are related to each other by a single exchange.}
\label{fig:site}
\end{figure}

\begin{proposition}\label{prop:trailsrel}
Let $T_1$ and $T_2$ be any two trails of a 1- linkoid diagram $K$. Then there exists a finite sequence of double or single exchanges taking $T_1$ to $T_2$. If $K$ is a knot-type 1-linkoid diagram then $T_2$ is obtained from $T_1$ only by a number of double exchanges. Otherwise, both types of exchanges may be utilized to obtain $T_2$ from $T_1$. 

\end{proposition}

\begin{proof}

 It is clear that any two trails of $K$ differ at a finite number of smoothing sites.

We proceed by induction to first show that if $K$ is a knot-type 1- linkoid diagram then $T_2$ is obtained from $T_1$ only by a number of double exchanges.  

In Figure \ref{fig:trailxhange} we see that a double exchange is required to take the (knot-type) trail  to the other one given on the top row. This constitutes the initial inductive step as the trail is the simplest trail with the least number of sites in it.

Suppose $T_1$ and $T_2$ are two trails of a knot-type 1-linkoid diagram. Assume that $T_1$ and $T_2$ differ from each other at (say) $n > 0$ sites. Resmooth one site of $T_1$ to form a diagram $t$. $t$ has two components, one of which is a closed loop $\lambda$. The loop $\lambda$ must have another site that differs from $T_1$ or else when we resmooth all $n$ sites, we would still have an extra component corresponding to $\lambda$ or a part of $\lambda$. Therefore we can resmooth at a site on $\lambda$ to form a new trail $T_3$ that differs from $T_2$ at $n - 2$ sites. Thus we can proceed by induction to complete the proof for knot-type trails.

Let $K$ be a proper 1-linkoid diagram. Choose one of the smoothing sites of $T_1$ that differs from $T_2$. If resmoothing of that site keeps $T_1$ connected then this is a single exchange, otherwise we can make a double exchange as we did above. Thus in any case we can reduce the number of differing sites by one or two, and proceed by induction.



\end{proof}

\begin{figure}[H]
\centering
\includegraphics[width=.86\textwidth]{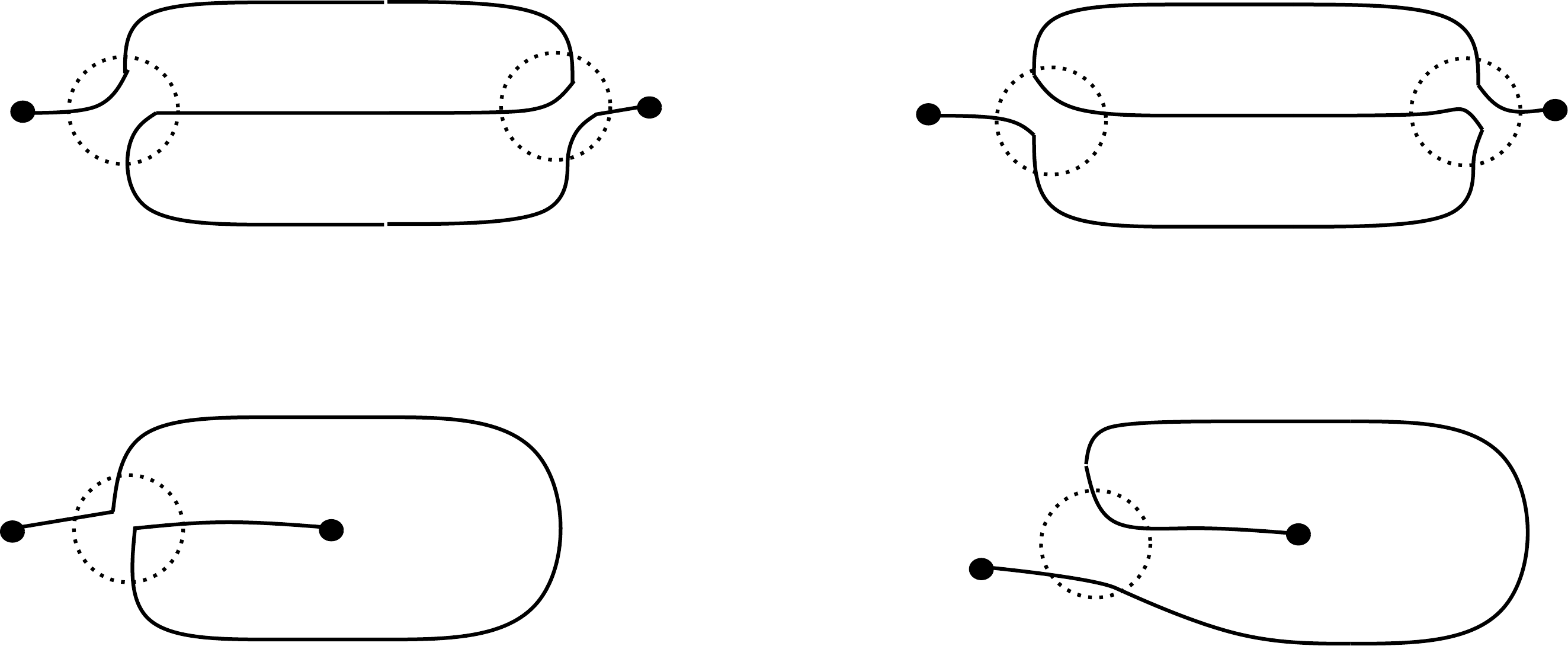}
\caption{Simplest trails that are related to each other by a double and single exchange. }
\label{fig:trailxhange}
\end{figure}

\begin{proposition}\label{prop:statemoves}
Any clock state of a 1-linkoid universe that is not a curl or curl composite admits a number of clock moves. 

\end{proposition}

\begin{proof}
There is a one-to-one correspondence between clock states and trails of a 1-linkoid universe (see Corollary) \ref{cor:statesandtrails}, and any two trails are related to each other by single or double exchanges by Proposition \ref{prop:trailsrel}. A single or double exchange on a trail corresponds to a clock move in a state corresponding to the trail. We can deduce that any clock state of a 1-linkoid universe which determines a trail admitting an exchange on it, admits a clock move on it. 

\end{proof}

\subsubsection{Shell compositions and derivation}

In this section, we define an operation on 1-linkoid universes that is called \textit{derivation}. The derivation operation decomposes the universe of a 1-linkoid diagram into its basic ingredients that we call \textit{shells}. Later in the paper, the derivation and shells will be utilized for showing the existence of  extremal states for a knotoid universe. 

\begin{definition}\normalfont
A \textit{shell} is a circle that sits on an edge, intersecting the edge transversally at one or two points, as shown in Figure \ref{fig:shell}.  The vertices that are connected by the edge are called \textit{endpoints}. A shell consisting of only an edge without a circle is called an \textit{empty shell}. Each shell may contain  \textit{cusp interactions} between its edges and other edges that are accessible. In Figure \ref{fig:shell}, we show examples of shells that do not contain any cusp interactions. Notice that the first shell in the figure is an empty shell. 
Note also that each shell can be considered as a $1$-linkoid universe.

\end{definition}
\begin{figure}[H]
\centering
\includegraphics[width=.75\textwidth]{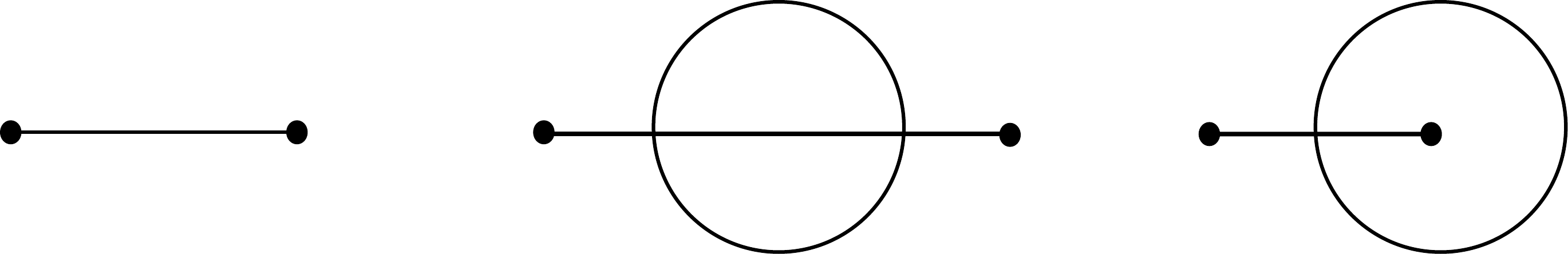}
\caption{Some basic shells.}
\label{fig:shell}
\end{figure}

Let  $s$ be the shell shown in the middle of Figure \ref{fig:shell}. The edges of $s$ that connect the trivalent vertices are called   \textit{top-line}, \textit{mid-line} and \textit{bottom-line}, respectively with respect to the top to bottom direction of the plane. If $s$ is of the type shown as the last shell in Figure \ref{fig:shell}, then the \textit{mid-line} edge is the edge that connects the trivalent vertex to one of the endpoints surrounded by the loop at the trivalent vertex. In this case, the top-line is considered to be the part of the loop that remains above the mid-line and the bottom-line is considered to be the part of the loop that remains below the mid-line.

\begin{definition}\normalfont

A shell composition $s$ is the graphical union of a finite number shells $\{s_i\}_{i}$, each of which is placed on the basic parts of the other,  possibly with cusp interactions between edges of shells. We call each shell $s_i$ a \textit{rider} on $s$. 

 A shell composition without any cusp interactions is called a \textit{pure shell composition}. The shells depicted in Figure \ref{fig:shellcomp} are examples of pure shell compositions.  
\end{definition}

\begin{figure}[H]
\centering
\includegraphics[width=.7\textwidth]{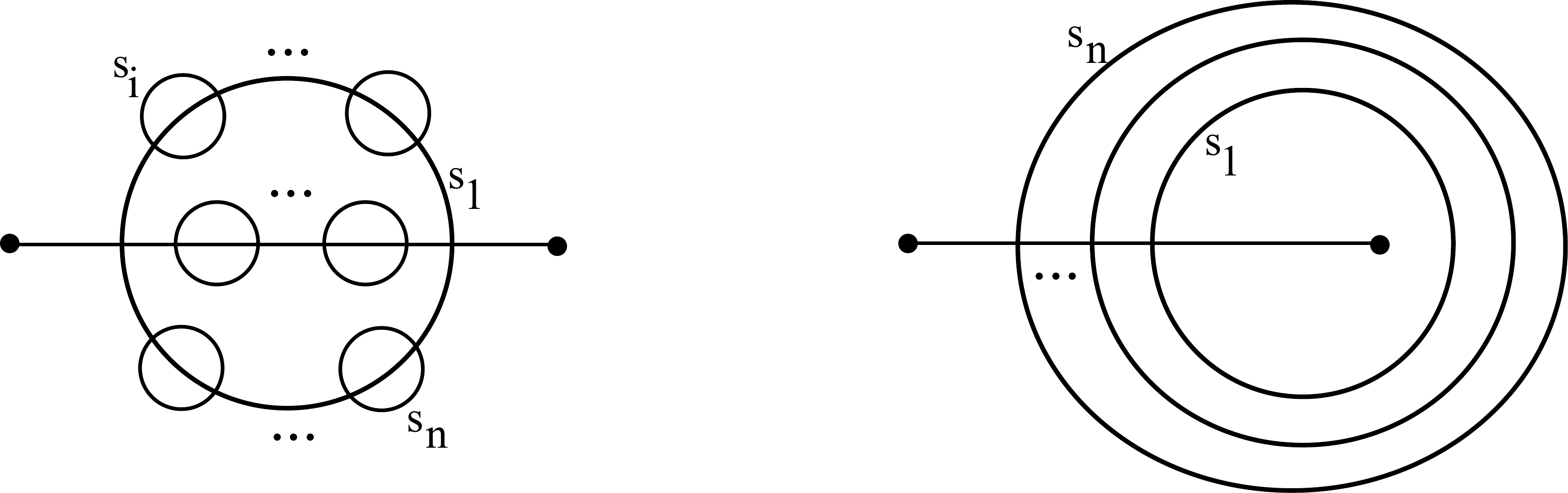}
\caption{Pure shell compositions.}
\label{fig:shellcomp}
\end{figure}

\begin{definition}\normalfont
A \textit{nugatory} vertex on a 1-linkoid universe is a vertex that admits only one type of smoothing that retains connectivity.

\end{definition}

\begin{definition}\normalfont

A \textit{curl} is a knot/link or a 1-linkoid universe that consists of only one vertex. A \textit{curl composite} is a knot or knotoid universe that is a connected sum of curls.

A crossing on a curl is said to be \textit{of curl type}, and a 1-linkoid diagram consisting of only curl type crossings is called a \textit{curl composite}.
\end{definition}

\begin{note}
It is easy to see that a 1-linkoid universe with only nugatory vertices be obtained from a number of trivial  knot and a knotoid diagram by adding curls (by applying RI-moves) to them.

\end{note}

\begin{definition}\normalfont
The \textit{composition}  $U_{K_1}$ $\bigoplus$ $U_{K_2}$ of two 1-linkoid universes $U_{K_1}$, $U_{K_2}$ is defined similarly with the composition of two knotoid diagrams: The head of $U_{K_1}$  is spliced to the tail of $U_{K_2}$. 
\end{definition}

Note that if the head of $U_{K_1}$ lies in a bounded region then $U_{K_2}$ is inserted in the bounded region. The composition operation $\bigoplus$ on universes is associative but not commutative in general.

\begin{definition}\normalfont
A 1-linkoid universe is \textit{prime} if it cannot be decomposed into non-trivial universes, that is $U$ is not of the form $U_1 \bigoplus U_2$, where $U_i$ is a non-trivial universe, for $i=1, 2$. Otherwise it is called a \textit{composite} knotoid universe.
\end{definition}

\begin{definition}\normalfont
 Let $e$ be an edge of a 1- linkoid universe $U_K$, and $q$ be an interior point on $e$.  We remove $q$ from $e$ and add two vertices at the resulting two boundary points in $U_K - \{e\}$. This operation is called the \textit{removal} of an edge.

\end{definition}
Note that removal of an edge adds a new knotoid component to the knotoid universe, which implies that $U_{K}- \{ e\} $ is a 2-linkoid universe.
\begin{definition}\normalfont
An edge of a 1-linkoid universe is called an \textit{endpoint edge} if it is incident to an endpoint of the universe.  An edge $e$ of a composite 1-linkoid universe $U_K$ is called a \textit{connecting edge} if the removal of $e$ is the disjoint union of two non-trivial knotoid universes that can be prime or composite. 
An edge that is incident to an interior region of a 1-linkoid universe is called an \textit{interior edge}. An edge that is incident to the exterior region where there is a star marker is called a \textit{boundary edge}.

\end{definition}

Notice that any interior edge bounds two distinct regions if it is not an endpoint edge.
\begin{definition}\normalfont
Let $U_1$ and $U_2$ be two 1-linkoid universes. Assume that $U_2$ is a knot-type knotoid universe and $p$ be an interior point of an edge of $U_1$. Let $U= U_1 \bigoplus [U_2,p] = U_1 \bigoplus [U_2]$ denote the universe obtained by replacing an interval in $U_1$ containing $p$ by a copy of $U_2$ by identifying the endpoints of $U_2$ by the endpoints of the interval  and in a way that $U_2$ does not intersect the rest of $U_1$. We call this operation the \textit{insertion} of $U_2$ into $U_1$, and
$U$ decomposes into a \textit{carrier} $U_1$ and a \textit{rider} $U_2$. 

Notice that $U_1 \bigoplus [U_2,p] = U_1 \bigoplus U_2$ if $U_2$ is placed at a point that belongs to an endpoint or connecting edge of $U_1$.
\end{definition}

\begin{definition}\normalfont
A knotoid universe is called \textit{simple} if it is prime and has no riders. A knotoid universe that is either composite or containing riders is called a \textit{compound} universe.
\end{definition}

Now we are ready to define the derivation operation.

\begin{definition}\normalfont

Let $U_K$ be the universe of a starred 1-linkoid $K$. We assume the star marker is located at the exterior region.
We define the \textit{derivation} on $U_K$, denoted by $D$ in the following way. 
\begin{enumerate}[i.]
\item  Let $U_{K}$ be simple.  If $U$ is a curl, then the operation $D$ smooths the vertex of the curl and then connects the boundary edges that lie in the opposite local sides of the vertex as in Figure \ref{fig:derivationsimple}. The derivation of a curl is an empty shell with an interaction site containing a pair of cusps.
If $U_K$ is not a curl , then $D$ smooths all its vertices that lie on the boundary edges except the vertex (or vertices) that is  incident to the endpoint that lies in the exterior region. See Figure \ref{fig:derivation} as an example.

The derivation of  $U_K$ is then one of the shells shown in the middle and the rightmost part of Figure \ref{fig:shell} with a number of interaction sites in the form of a cusp at the smoothed crossings. Note that the derivation of $U_K$ can be also thought as a 1- linkoid universe with a number of cusp interactions.

\item
Let $U_K$ be compound. \\
 If $U_K = U_{K_1} \bigoplus U_{K_2}$, then $D(U_K) = D(U_{K_1}) \bigoplus D(U_{K_2})$.\\
 If $U_K=  U_{K_1} \bigoplus [U_{K_2}]$,  then $D(U_K) = D(U_{K_1}) \bigoplus [D(U_{K_2})]$.

\end{enumerate}
See Figures \ref{fig:derivationsimple} and \ref{fig:derivation} for derivations applied on simple universes. 
\end{definition}
\begin{figure}[H]
\centering
\includegraphics[width=.4\textwidth]{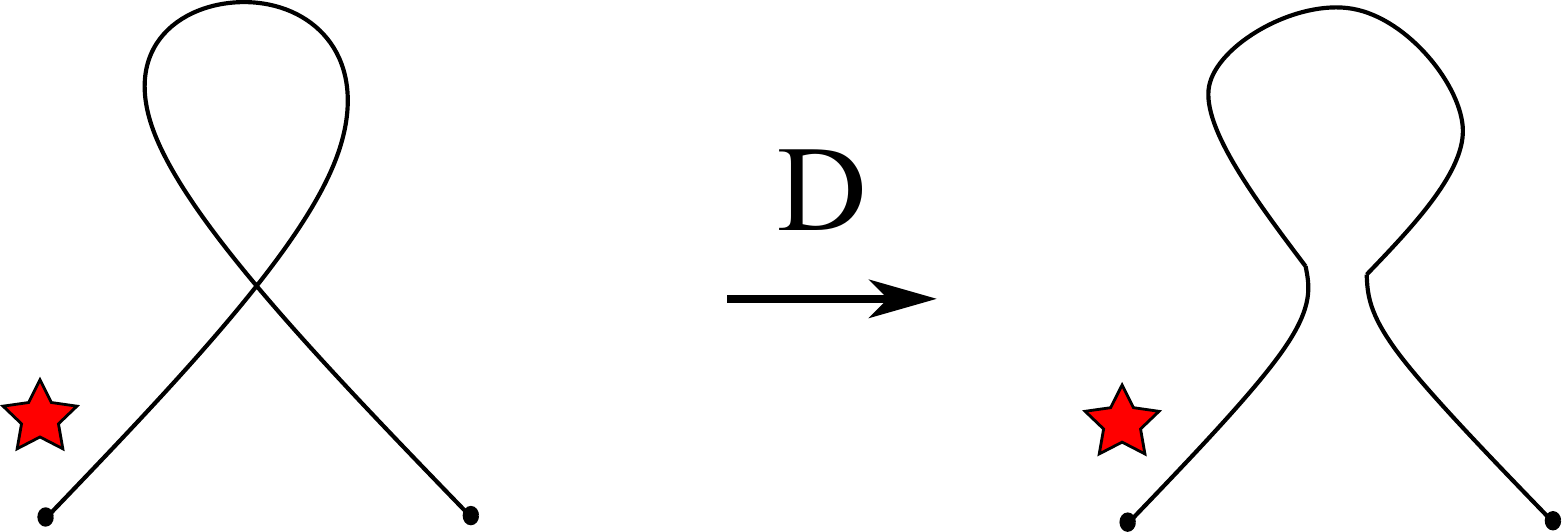}
\caption{The derivation on a curl}
\label{fig:derivationsimple}
\end{figure}

\begin{definition}\normalfont

The composition $D_1\circ D_2$ of two derivations $D_1$, $D_2$ on a non-trivial 1-linkoid universe $U_K$ is defined by smoothing the boundary vertices of $U_K$ firstly and then transforming the star from the exterior region to the adjacent region of the resulting universe. The new starred region is considered to be the exterior region for the knotoid universe obtained and its boundary vertices are determined accordingly for the second derivation. See Figure \ref{fig:derivation} that illustrates the composition of two derivations of a knotoid universe. 

\end{definition}

\begin{definition}\normalfont
Let $U_K$ be a knotoid universe. The $n^{th}$ \textit{derivation} $D^n$ on $U_K$ is the composition of $n$ consecutive derivations applied to $U_K$.

\end{definition}

\begin{proposition}\label{prop:n}
For any knotoid universe $U_K$, there exists $n_{0} \in \mathbb{Z}_{ >0}$ and a shell composition $s$ such that $D^n (U_{K})=  s$ for all $n > n_{0}$ .

\end{proposition}

\begin{proof}
The derivation operation resolves the vertices of $U_{K}$ that lie on the boundary edges, except the vertices that are incident to the endpoints lying in the starred region, to obtain a simpler universe. Since there are a finite number of vertices, the recursive application of $D$ will stop at $n$ steps for some $n > 0$ when the resulting universe contains no vertices except the ones that are incident to the edges lying in the exterior region, that is, when the initial universe is transformed into a shell composition with interaction sites.  
\end{proof}

In Figure \ref{fig:derivation}, we see that two sequential derivations transform the knotoid universe $U$ into a Type II shell $s$ with a pair of cusp interactions located among its mid-line edge and also among mid-line and loop edge. It is clear that  $D^n (U)= s$ for any $n \geq 2$. 
\begin{figure}[H]
\centering
\includegraphics[width=.6\textwidth]{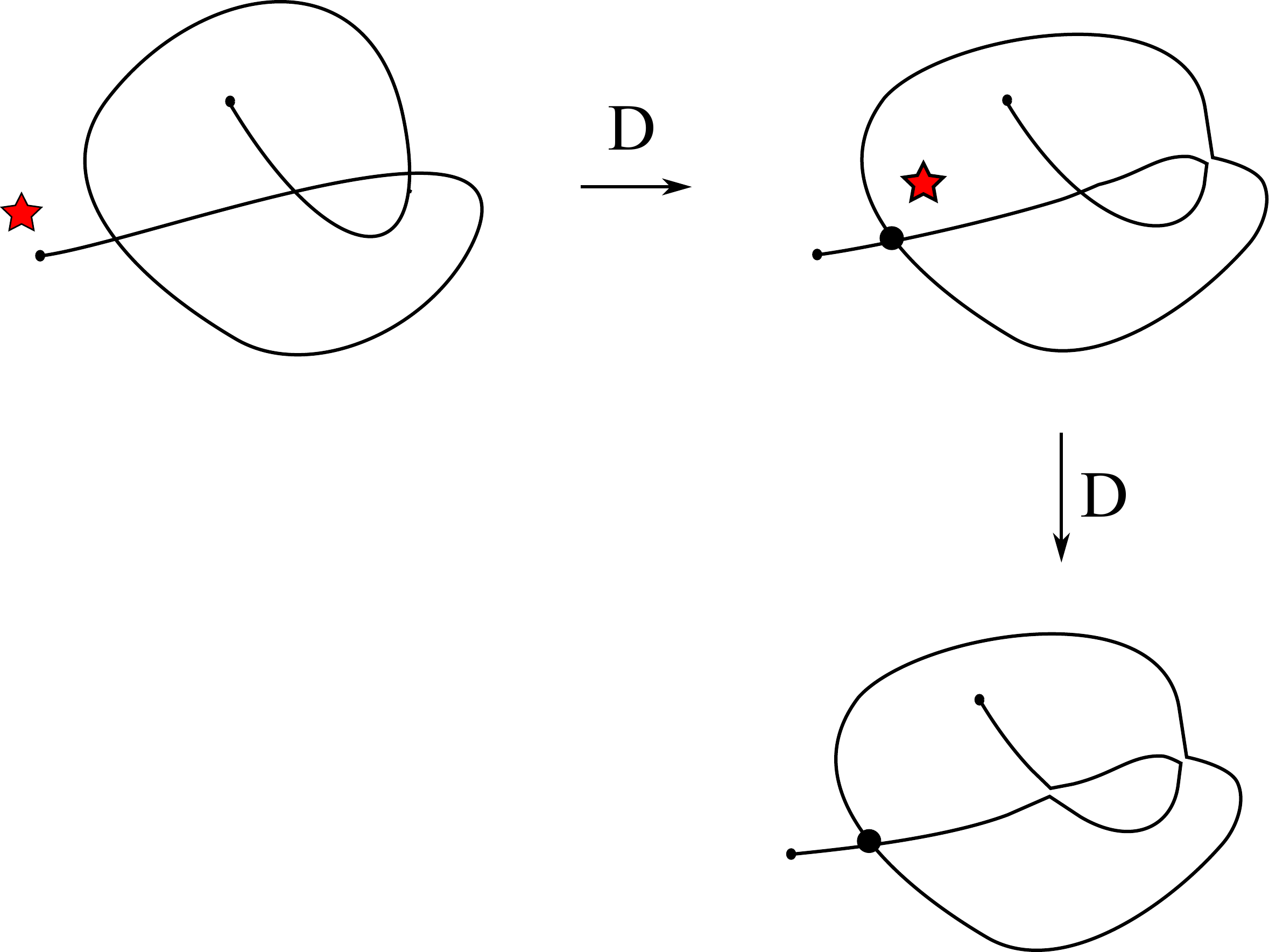}
\caption{The composition of two derivations on $K$.}
\label{fig:derivation}
\end{figure}

\begin{definition}\normalfont

 The \textit{closure} $c(s)$ of a shell composition $s$ is obtained by replacing each cusp interaction with a flat crossing as shown in Figure \ref{fig:closure}.

\end{definition}
\begin{figure}[H]
\centering
\includegraphics[width=.4\textwidth]{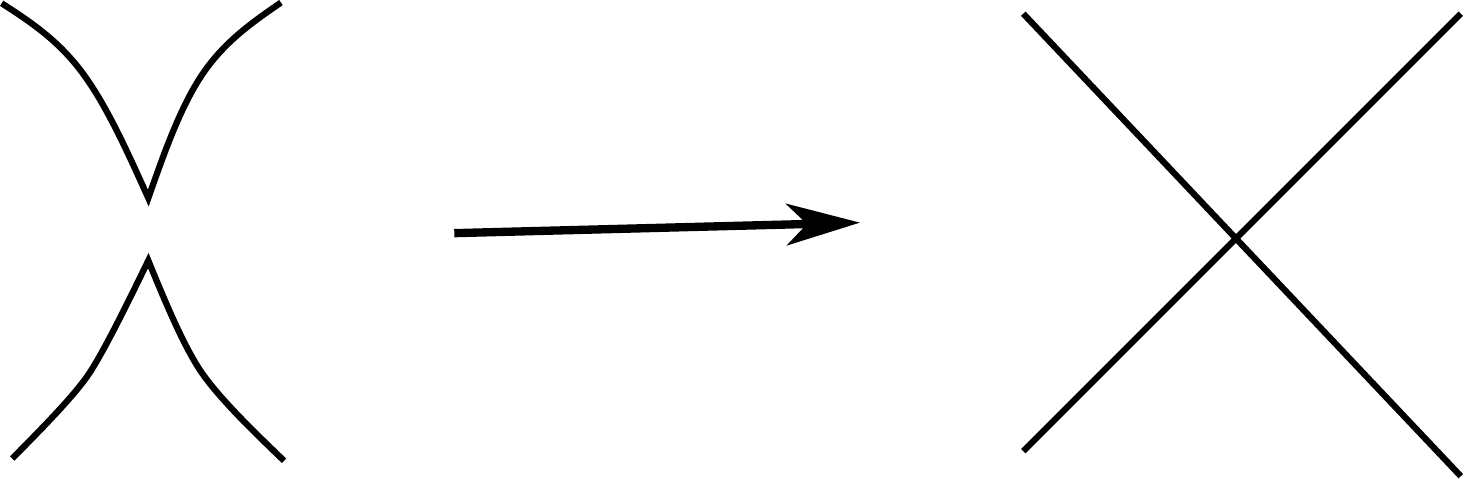}
\caption{Closure of a cusp interaction site}
\label{fig:closure}
\end{figure}

It is clear that closing each interaction site turns any shell composition into a unique 1-linkoid universe. In fact, the closure  is a well-defined mapping from the set of all shell compositions to  the set of all 1-linkoid universes. It is straightforward to see that if $s$ is one of the  type 0, I and II shells, without any cusp interactions, given in Figure \ref{fig:shell} then $c(s)=s$ and $D \circ c (s) = s$.  However, in Figure \ref{fig:nonid}, we observe that the derivation does not necessarily constitute a left inverse for the closure operation. 

\begin{figure}[H]
\centering
\includegraphics[width=1\textwidth]{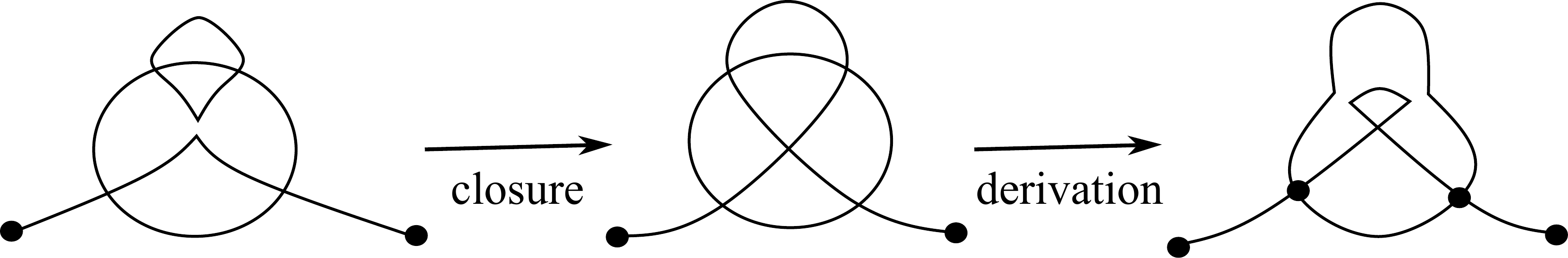}
\caption{}
\label{fig:nonid}
\end{figure}


\begin{definition}\normalfont \label{defn:int}
In the course of derivation, vertices that are smoothed are the ones that are incident to boundary edges or lying on a curl.  We call the resulting interactions \textit{boundary} and \textit{curl type} interactions, respectively.
\end{definition}
Note that the reason that $D \circ c (s) \neq s$ is that the cusp interaction in the shell composition given in Figure \ref{fig:nonid} is neither a boundary not a curl type interaction.






\begin{definition}\normalfont \label{defn:interaction}
Let $K$ be a $1$- linkoid diagram.
A shell composition $s$ is said to \textit{underlie} $K$ if   $c(s) = K$ and $D^{n} (K) = s$, for some $n \in \mathbb{Z}_{>0}$ and  $s$ is a shell composition with $D^{m}(s) =s$ for every $m >n$.


\end{definition}

\begin{proposition}\label{prop:shellcomp}
Every 1-linkoid diagram admits a unique underlying shell composition.
\end{proposition}

\begin{proof}
By Proposition \ref{prop:n}, we know that there is a sequence of derivations that transforms a 1- linkoid diagram into a shell composition.  Then by construction, such shell composition may contain interactions only of the types given in Definition \ref{defn:int}.

\end{proof}

\begin{definition}\normalfont
Let  $s$ be the shell composition underlying a 1-linkoid diagram $K$. The endpoint of $s$ that corresponds to the tail of $K$ is called the \textit{tail} of $s$. The endpoint of $s$ that corresponds to the head of $K$ is called the \textit{head} of $s$.

\end{definition}



\subsubsection{An algorithm for obtaining a clocked state: Shelling algorithm}

Let $s$ be the shell composition underlying a 1- linkoid diagram $K$. We can always assume that the tail of $s$ lies on the left-hand side of every shell of $s$ and  at the exterior region.

We place markers as shown in Figure \ref{fig:int} at  interaction sites in $s$. Let us explain the rule of placing the markers.
\begin{figure}[H]
\centering
\includegraphics[width=.6\textwidth]{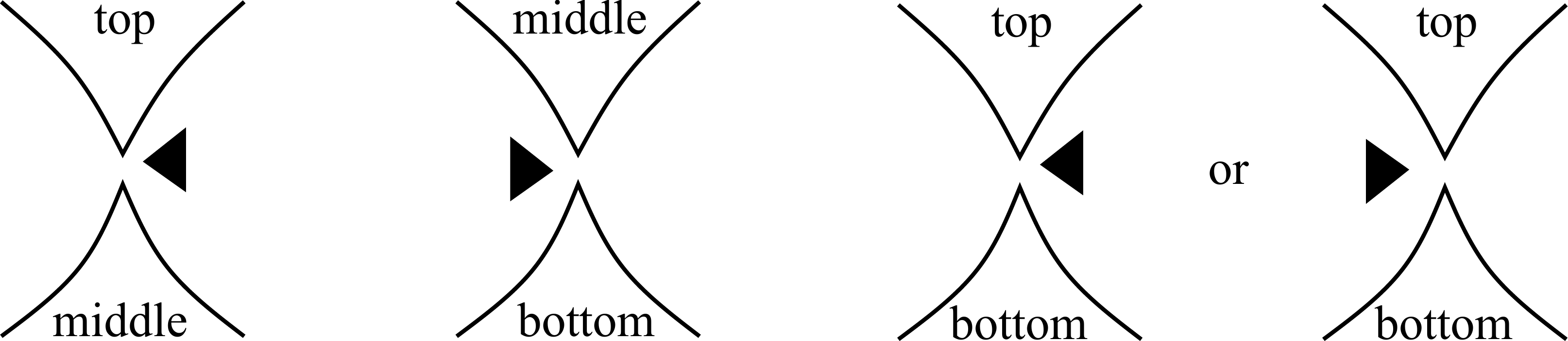}
\caption{Placement of state markers at interaction sites.}
\label{fig:int}
\end{figure}

\begin{enumerate}

\item If an interaction site lies between top-line and mid-line, we place a marker at the right-hand side of the interaction site.

\item If an interaction site lies between bottom-line and mid-line, we place a marker at the left-hand side of the interaction site.
\item If an interaction site lies between top-line and bottom-line, we place a marker either  at the left-hand side or right-hand side of the interaction site. 
\item If an interaction site is a self-interaction site of either top-line or mid-line or bottom-line, then we place a marker at the \textit{inside} of the interaction site that corresponds to the bounded region enclosed after the closure of the site.
\end{enumerate}

 This placement of the state markers on $s$ induces a state in the closure of $s$, that is on $K$, by considering  markers added at interaction sites of $s$ as the state markers at crossings in $K$.  See Figure \ref{fig:shellex} for an example. Note that the interaction site that is incident to the exterior region can be considered to lie between either top-line and mid-line or bottom-line and mid-line. 
 
 In Theorem \ref{thm:existence} we show that the induced state always corresponds to a clocked state of $K$. Note that in the shelling algorithm if we swap the directions for the markers that are placed, we obtain a counter-clocked state for $K$, and in fact the algorithm always produces a counter-clocked state as also follows by Theorem \ref{thm:existence}. 
 
\begin{figure}[H]
\centering
\includegraphics[width=.6\textwidth]{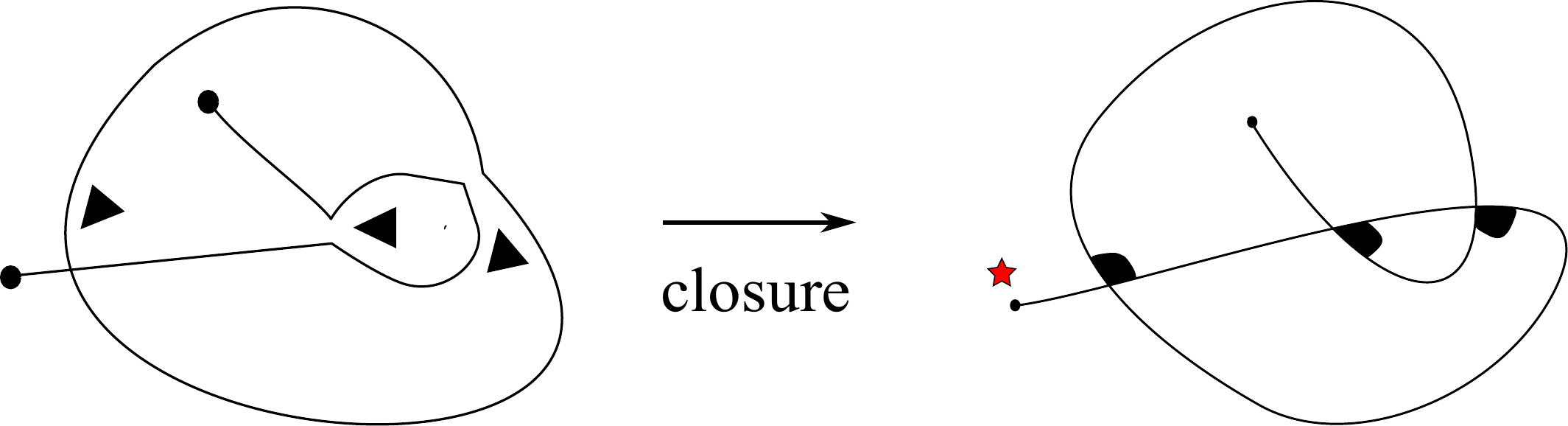}
\caption{Shelling of $	K$ and the corresponding state.}
\label{fig:shellex}
\end{figure}

\begin{theorem}[Existence of extreme states] \label{thm:existence}

Every 1-linkoid universe admits a clocked and a counter-clocked state.
\end{theorem}

\begin{proof}
The proof proceeds by induction on the number of vertices of the universe of a 1-linkoid diagram.

If $U_K$ is the trivial universe then the statement holds directly since there is no state that can be assigned to $U_K$ . 

If $U_K$ is a curl or a curl composite, then it is clear that there is a unique state of $U_K$ which is obtained by placing state markers at  bounded regions of curls in $U_K$. This state is conventionally assumed to be one of the extremal (clocked or counter-clocked) states of $U_K$.

If $U_K$ admits one or two vertices and it is not a curl composite, then it is not hard to see the shelling algorithm induces a clocked state.  See Figure \ref{fig:split} for some of the 1-linkoid universes with at most two crossings, with the shelling algorithm applied to them. 

\begin{figure}[H]
\centering
\includegraphics[width=.95\textwidth]{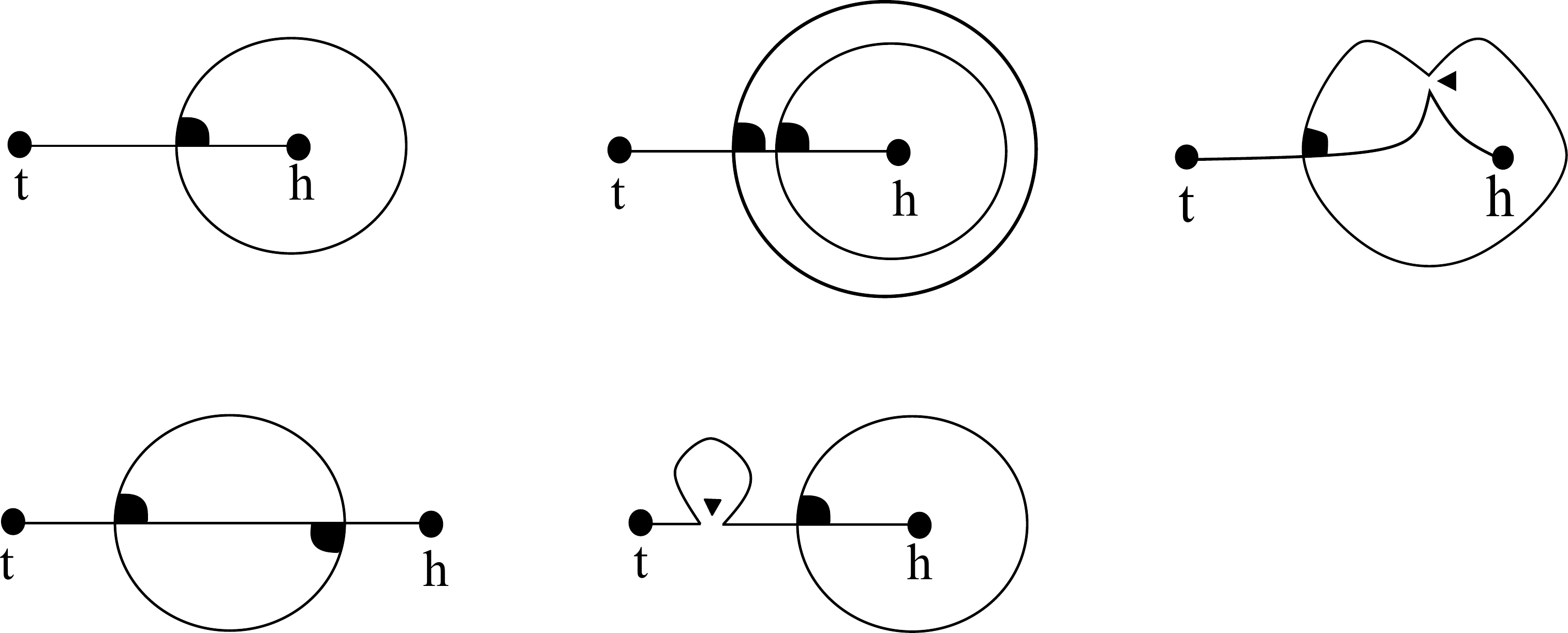}
\caption{Some shells that results from 1-linkoid universes with at most two vertices. Note that $t,h$ denote the tail and head of the universes.}
\label{fig:split}
\end{figure}

Now assume that the statement holds for all 1-linkoid universes with less than or equal to $n$ vertices. Let $U_K$ be a 1-linkoid universe with $m$ vertices where $m = n+1$.

By deriving $U_{K}$ one time,  we obtain an outer shell whose mid-line structure is indeed a knotoid universe with a number of vertices that is less than the number of vertices of $U_K$ and with a number of interaction sites that lie on top-line, mid-line or bottom-line of the resulting shell.  By the induction hypothesis, we know that there is a clocked state for the knotoid universe that forms the mid-line part. We place the state markers at the vertices of the knotoid universe so that we get a clocked state. 

We then place state markers at the resulting interaction sites according to the rules of the shelling algorithm. We observe that there is no available clock move in the counter-clockwise direction that results from an interaction of these markers with the markers that lie on vertices of the mid-line knotoid universe. The reason for this is that any clock move in the counter-clockwise direction would take a (state) marker at an interaction site to the starred region, and the starred region does not get any state marker by the construction.

This observation completes the proof of the existence of extremal states.

\end{proof}

 \begin{lemma} (Removal Lemma)\label{thm:removal}
  Let $U$ be a 1 - linkoid universe and $s$ be a clocked state of $U$.  
 Then any state obtained by smoothing out a crossing of $U$ in the direction of the state marker at that crossing, is a clocked state.
 \end{lemma}
 
 \begin{proof}
In Figure \ref{fig:removeya!} we illustrate two generic cases of a clocked state of $U$ locally, one of which contains an endpoint adjacent to the crossing removed.  On the right-hand side pictures of the figure we see local parts of a smaller universe with the crossing $c$ removed in the direction of the given state marker at $c$. Note that the regions $R_i$ and $R_j$ are distinct regions as noted in the proof of Proposition \ref{prop:Jordan}.

Upon removal of $c$, the edges $e_1$ and $e_2$ get connected with the $e_4$ and $e_3$, respectively, and the region $R_i$ gets connected with the region $R_j$. We see that the region $R_j$ receives a marker at the bottom most crossing and so all the regions of the resulting universe $U - \{c\}$ receive a state marker. If $R_j$ does not get a marker at the bottom most crossing as in the pictures, it is guaranteed that $R_j$ gets marked at another crossing because it is already a region of the clocked state of $U$ we have begun with. This means that the remaining state markers determine a state for the universe $U - \{c\}$. 

By Proposition \ref{prop:statemoves}, any state admits a clockwise or counter-clock wise clock move  since any state determines a trail and any trail admits a single or double exchange sites. 

If there is a clock move in the counter-clock wise direction available for the resulting state, then it is either a clock move that moves a single or a pair of state markers that lie outside of this local part of the universe  $U - \{c\}$, or it moves a state marker of $U - \{c\}$ that appears in the figure. Both cases are not possible since any such clock move would induce a clock move in the counterclockwise direction in the initial state. Therefore, the resulting state of $U - \{c\}$ is a clocked state. 
 
This figure can be considered as a generic picture of a local configuration of any clocked state and the observation we made above. 

\begin{figure}[H]
\centering
\includegraphics[width=.6\textwidth]{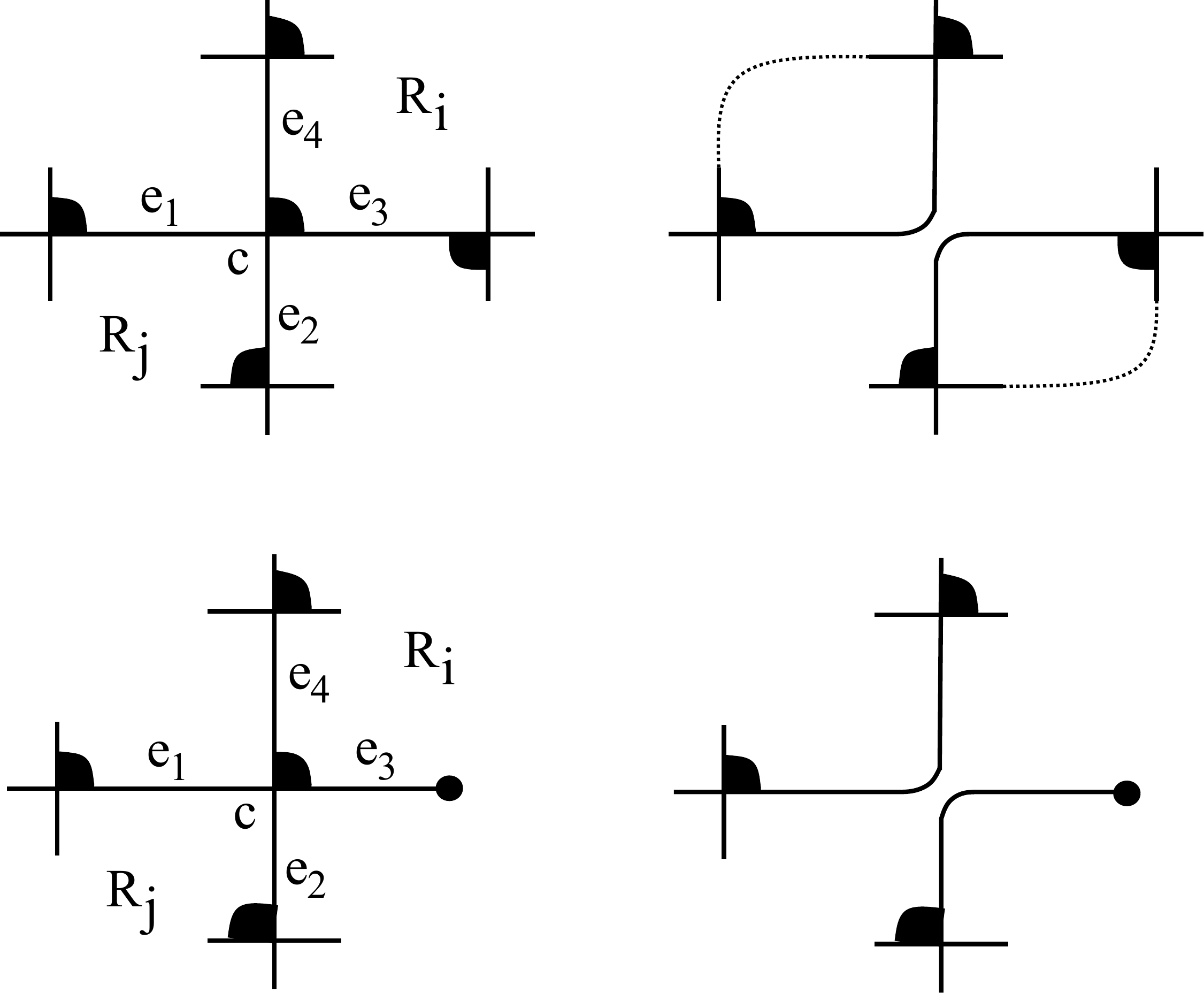}
\caption{Two types of a local picture of a clocked state. }
\label{fig:removeya!}
\end{figure}

  \end{proof}

\begin{theorem}[Uniqueness of Extreme States]\label{thm:unique}
The extreme states (clocked and counter-clocked) of a 1- linkoid universe are unique. The extreme states are obtained  by the shelling algorithm given  in the proof of the existence of clocked states.  \end{theorem} 
\begin{proof}
We prove the uniqueness of the extremal states (clocked and counter-clocked states) by induction on the number of vertices of universes. 

Assume first a 1-linkoid universe that is either a curl or a curl composite.  That is, the underlying trail is of curl type. 
We already discussed in the proof of Theorem \ref{thm:existence} that such a universe  admits a unique state that is both clocked and counter - clocked by convention.

Now assume that any 1-linkoid universe with at most $n$ crossings admits a unique clocked and counter-clocked state.  Let $U$ be a 1-linkoid universe with $n+1$ crossings, and $s$ be an extremal state of $U$.

 Without loss of generality, let us assume $s$ is a clocked state. By applying a single derivation on $U$, we obtain a shell composition. The midline of the resulting shell composition is clearly a universe itself and has less crossings than $U$. By the removal lemma we know that removing a crossing from $U$ induces a clocked state  $\tilde{s}$ on the midline universe.
  Then by the induction hypothesis, it follows that $\tilde{s}$ is the unique clocked state on the midline universe.  Putting back the removed crossings,  there is only one possible placement for the markers at the crossings to have a clocked state and this state is indeed the initial state $s$ since the initial and terminal vertices (the vertices adjacent to the endpoints) are left during the derivation.

\begin{corollary}\label{cor:shelling}
The shelling algorithm induces the unique clocked trail.
\end{corollary}

\end{proof}

By the observations above, we can say that a clocked (and a counter-clocked) trail of a 1-linkoid universe that is not a curl has generically three main parts that we call \textit{top}, \textit{mid-line} and \textit{bottom} parts. See Figure \ref{fig:parts}.

\begin{figure}[H]
\centering
\includegraphics[width=.3\textwidth]{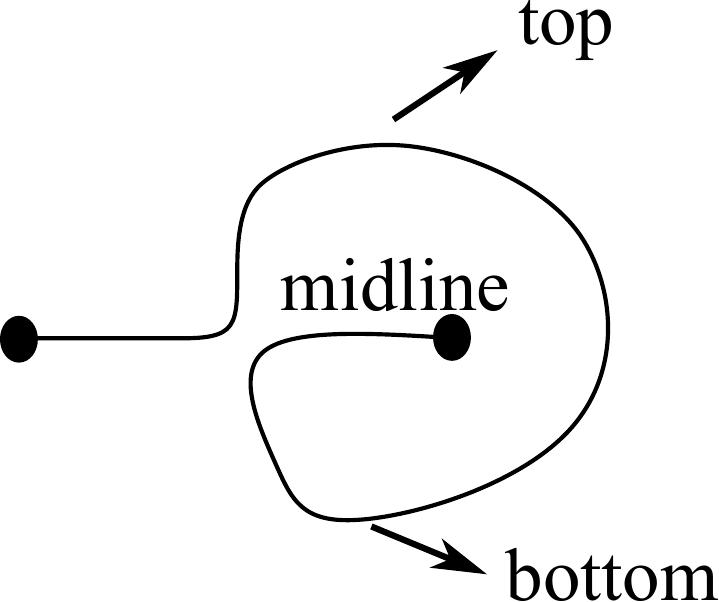}
\caption{Parts of a clocked trail}
\label{fig:parts}
\end{figure}


The following statement  tells that a clockwise or counterclockwise (single or double ) exchange on a trail of a 1- linkoid factorizes into a series of clock moves in the clockwise or counterclokwise direction, respectively on the corresponding states of the linkoid. The statement indeed generalizes the classical case where classical link diagrams are considered. We refer the reader to \cite{FKT} for the proof of the classical case statement.

%
 
\begin{definition}\normalfont

A \textit{trail-rider} is a trail $\tau$ that is inserted in an arc  $\alpha$ of some given trail $T$ for a  1-linkoid universe. The inserted trail $\tau$ may also have cusp interactions with cusps inserted in that trail $T$ outside of the arc $\alpha$. These extra cusps, if removed, will produce an insertion of a trail in $T$, so that the closure of the new larger trail $T \bigoplus \tau$ has a rider corresponding to this trail insertion. With the extra cusps, the closure  using $\tau$ may not necessarily produce a rider. We say that $\tau$ is an \textit{involved} trail-rider if no rider is produced in the closure.  \\

A trail is \textit{atomic} if it has no trail-riders.

 \end{definition}


 
 

\begin{theorem}\label{prop:clockprop}(Clock Theorem Extended)
Let  $T_1$ and $T_2$ be two trails of a 1-linkoid universe $U$ obtained by smoothing the vertices of $U$ according to the position of markers in any two states $s_1$ and $s_2$ of $U$, respectively. Assume that $T_2$ is obtained from $T_1$ by one exchange (clockwise or counterclockwise). Then, there is a sequence of clock moves applied in the clockwise direction (or counterclockwise direction, respectively) to transform the state $s_1$ to the state $s_2$. Clock moves are utilized as follows. In the statements below, riders will be always trail-riders, and we refer to markers in relation to the trails.


\begin{enumerate}
\item State markers that lie on trail riders on the top or bottom parts do not turn at all.
\item State markers that lie at the exchange sites turn by $90$ degrees.
\item If state markers do not lie at the exchange sites but  lie between the mid-line and the top or bottom then they turn a total of $180$ degrees clockwise.
\item If $A$ is a rider in the decomposition of the mid-line trail, and if $A$ has a cusp rider, then every marker on $A$ turns. Otherwise $A$ is \textit{uninvolved}, that is, state markers on it do not turn at all.
State markers at self-interaction sites of the mid-line trail that lie on an involved atom turn 360 degrees in the given clock direction.


\item Let $\tilde{C}$ be a composition of atoms that lie on the mid-line. Suppose that $A$ and $B$ are atoms that belong to $\tilde{C}$, and let $A$ ride on the atom $B$. If every marker on $A$ turns, then every marker on $B$ turns, too.
 \end{enumerate}
 
Below is an illustration of the extended clock theorem in Figure \ref{fig:swapp} where we enumerate the clockwise clock moves taking place during an exchange. 
 
\begin{figure}[H]
\centering
\includegraphics[width=1\textwidth]{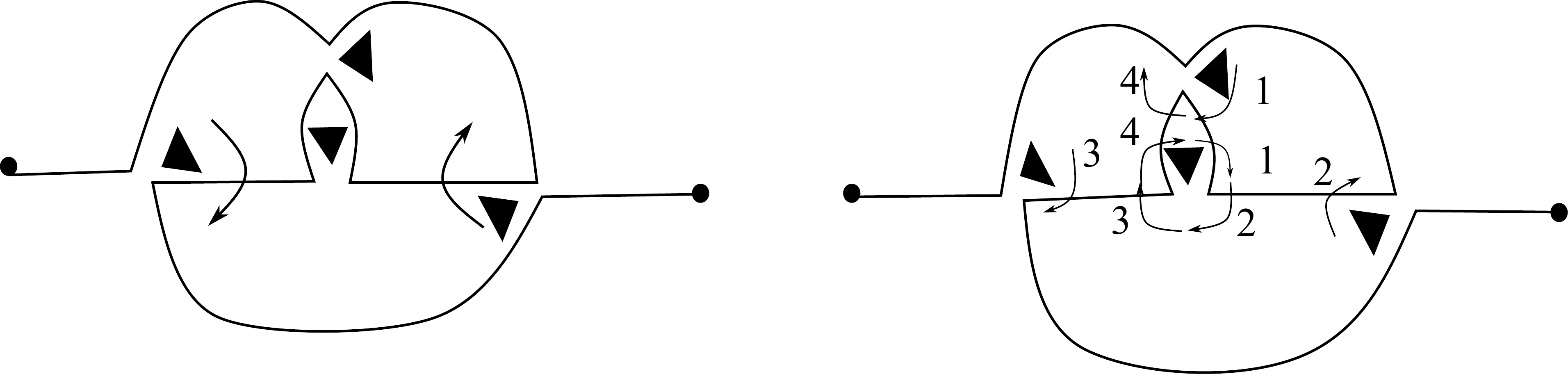}
\caption{}
\label{fig:swapp}
\end{figure}


\end{theorem}

\begin{proof}

We prove the statement by inducting on the complexity of the trail of a 1-linkoid diagram. The complexity we concern here is about the types and the number  of interaction sites of a trail.
The statement can be verified directly for 1-linkoid diagrams that have one or two crossings. 

Suppose first that $T_1$  admits some number of cusps between the mid-line trail and top or bottom lines as interaction sites. Clearly, some of these cusps are adjacent to exchange sites in $T_1$. An illustration for this is given in Figure \ref{fig:swapp} where we see only one interaction site between the mid-line and the top line, and it is adjacent to both exchange sites. 
 


We remove one of the  cusps that is adjacent to an exchange site.   Without loss of generality we can assume that the cusp we remove lies between the top-line and the mid-line. There are in fact two cases for the removed cusp: The removed cusp may also be adjacent to some another top-line cusp as in top left of Figure \ref{fig:cuspremoval} or it may be adjacent to a cusp that lies between bottom-line and mid-line, as illustrated in bottom left of Figure \ref{fig:cuspremoval}. The figures illustrate how starting moves in the non-removed configuration transform into moves in the removed configuration. In each case we see that if we know by induction that the process continues as in the hypothesis of the theorem for the removed configuration, then the process will extend to the configuration with the cusp plugged in. This completes the induction step.
 
\begin{figure}[H]
\centering
\includegraphics[width=1\textwidth]{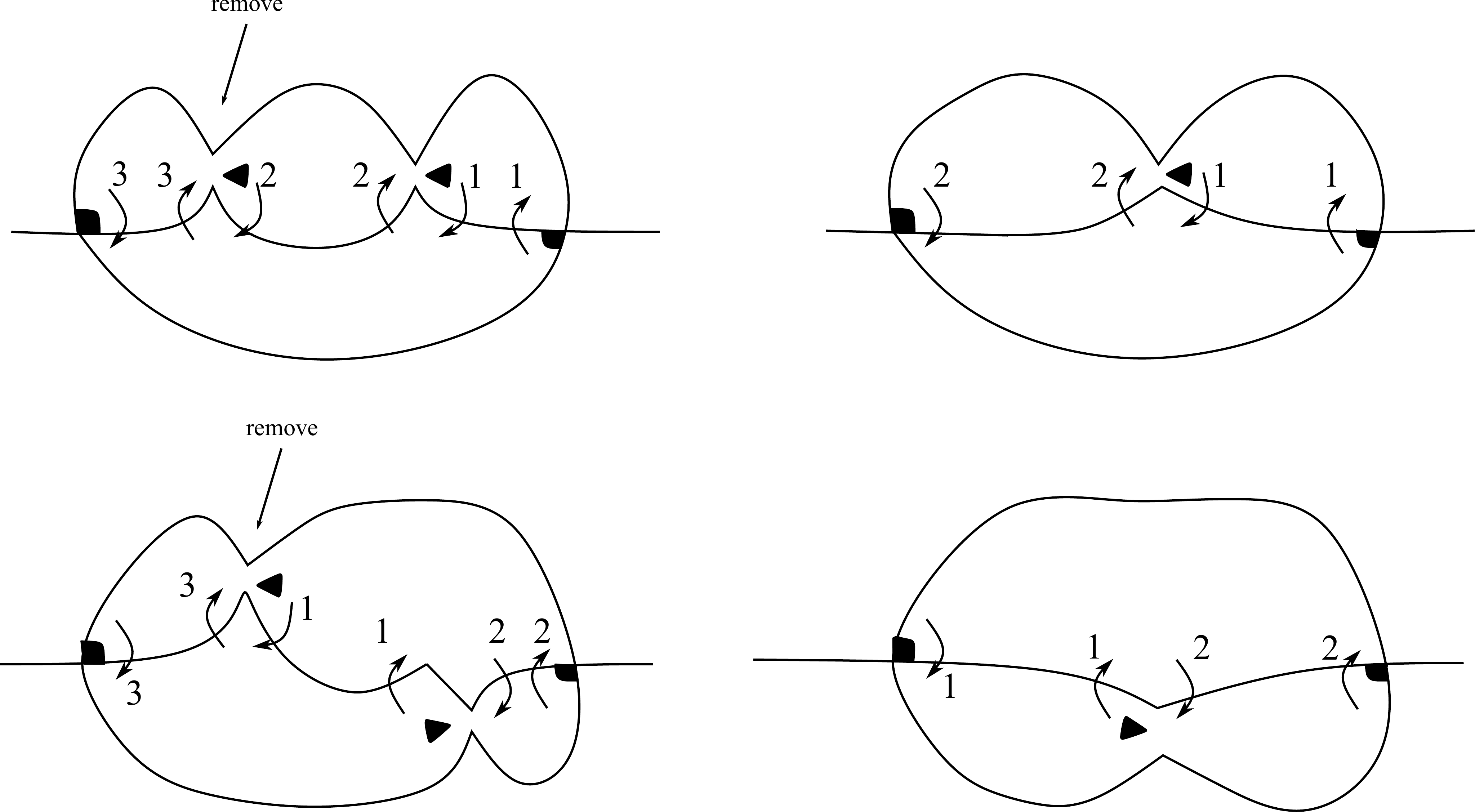}
\caption{}
\label{fig:cuspremoval}
\end{figure}


Assume now that the mid-line trail has self-interactions.
A self-interaction  on the mid-line is said to be \textit{internal} in the sense that it does not share any edge that is adjacent to regions that the exchange markers lie, otherwise it is \textit{of boundary type}.  Assume first that the mid-line trails have only boundary type self-intersections. Notice that the mid-line trail in this case consists of only curls and arborescent like riders of circular form, possibly with cusp interaction with top or bottom parts. 


One can see easily that none of the markers in the riders on the mid-line part turns during the clock exchange (in the factorization of clock moves) of the trails. That is, the riders are \textit{uninvolved}. The remaining markers that are indeed the exchange site markers turn 90 degrees in the clockwise direction. So the hypothesis applies directly. Assume that the hypothesis applies for every universe with the number of vertices less than or equal to $n$. Let $U$ be a universe with $n+1$ vertices whose mid-line consists of only boundary type self-interactions, and the riders on the mid-line are all uninvolved. Removing one of the self-interactions clearly induces a universe with $n$ vertices, whose mid-line part consists of all uninvolved atoms. In fact removing a boundary-type interaction gives a curl. Thus the induction hypothesis applies for the smaller universe and the factorization of this universe extends to a factorization of $U$.


Suppose now, that we have a trail whose mid-line still consists of only boundary type self-interactions (curls or circular riders) but a number of riders on the mid-line are involved (i.e. they interact with top or bottom part at a cusp). Choose a self-interaction site $s$ of the mid-line that lies on an involved rider and remove it.  If $s$ lies on a curl then the removal of $s$ makes the adjacent  rider  involved. The remaining markers on the small universe turn as the induction hypothesis indicates. When we plug $s$ in, these rotations induce a $360$ degrees turn of the marker at $s$. If $s$ lies on a circular rider that is involved, the removal of $s$ induces a curl that is involved. The induction applies to the smaller universe obtained by the removal of $s$ and plugging $s$ back makes $s$ rotate by 360 degrees in the clockwise direction.

Now assume that, the mid-line part has a number of internal self-interactions. Let $s$ be one of them. If $s$ lies on an uninvolved rider then the argument follows similarly as above upon the removal of $s$.
Suppose $s$ lies on an involved rider. If the adjacent crossings to the site $s$ remain involved when $s$ is removed, then by induction these markers turn as the hypothesis assume. When we plug in $s$, these rotations induce a $360$ degrees rotation of the marker at $s$. See Figure \ref{fig:involved}. 
\begin{figure}[H]
\centering
\includegraphics[width=.7\textwidth]{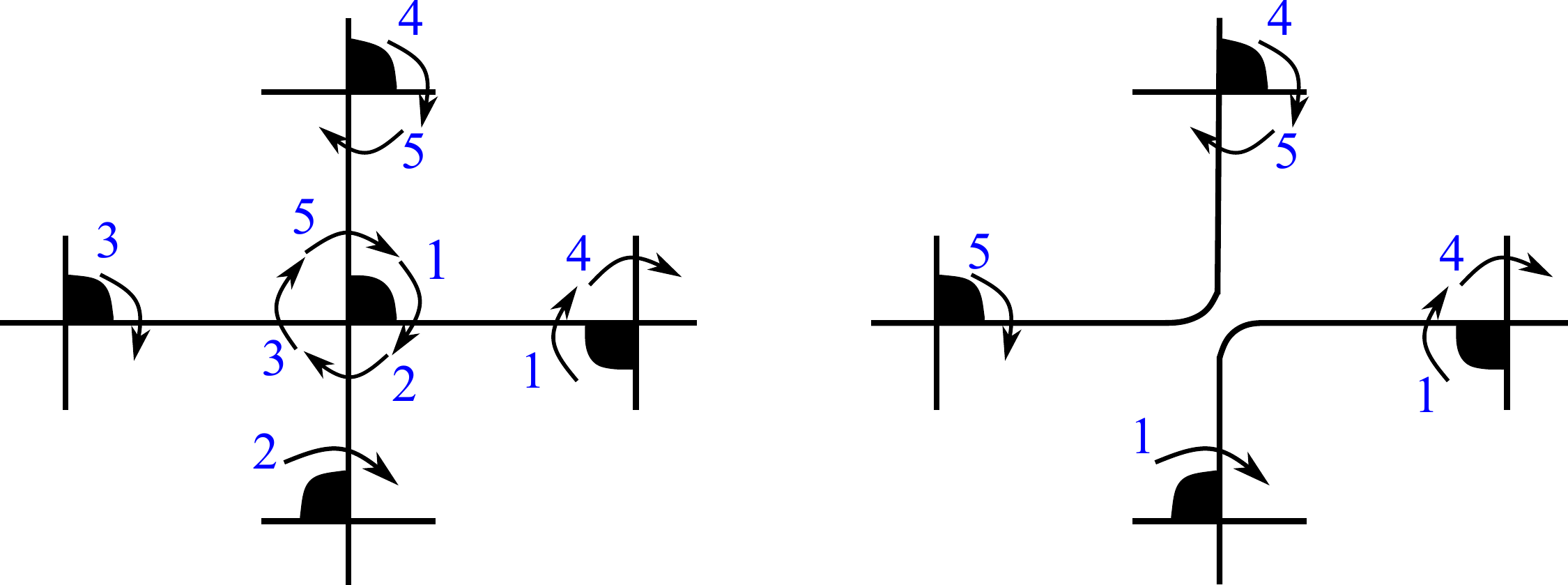}
\caption{}
\label{fig:involved}
\end{figure}

Note that Figure \ref{fig:involved}  illustrates a generic case where the removed site $s$ is adjacent to four vertices with markers. Other cases where the removed site is adjacent to two or three marked sites are shown in Figure \ref{fig:cases1} and Figure \ref{fig:cases2}.

\begin{figure}[H]
\centering
\includegraphics[width=.4\textwidth]{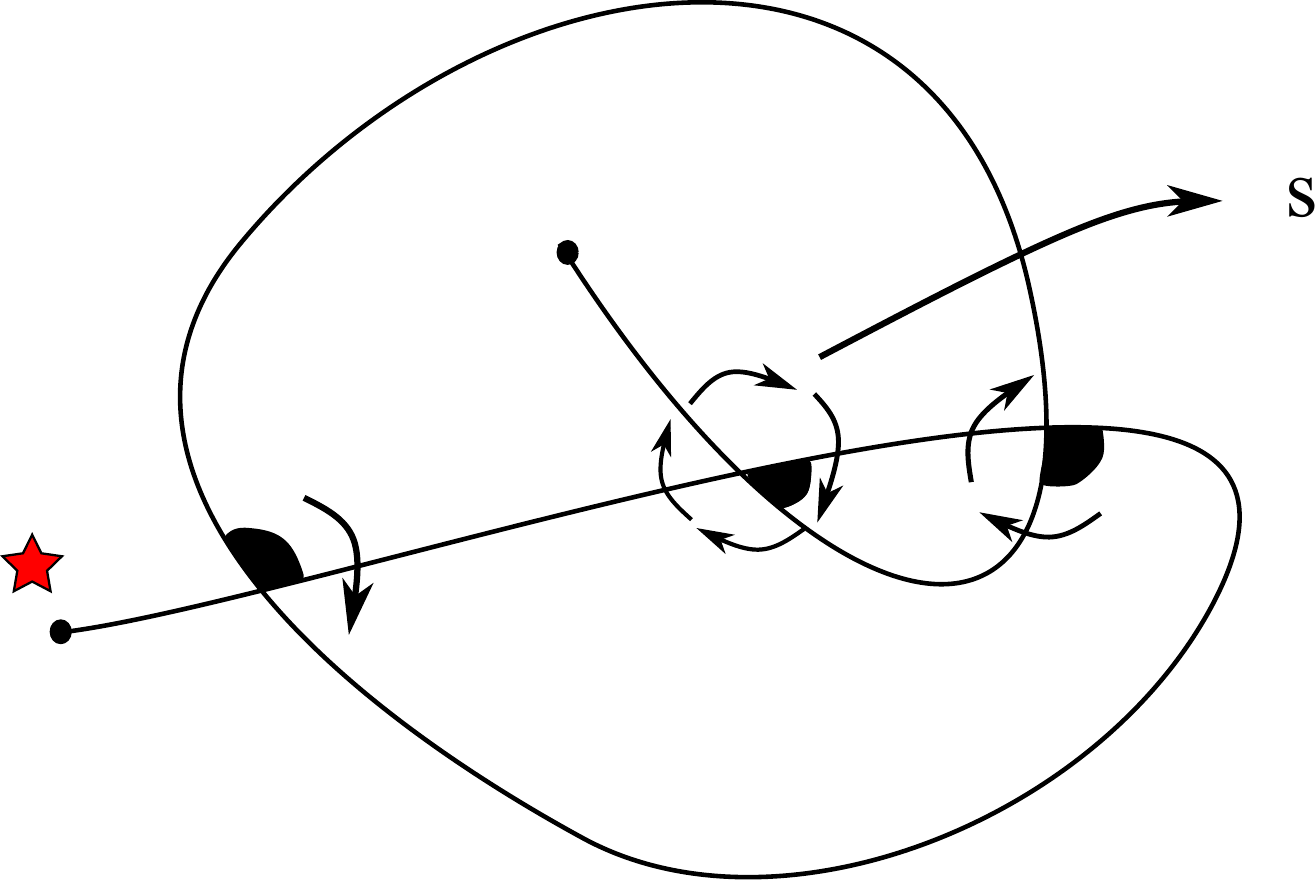}
\caption{}
\label{fig:cases1}
\end{figure}

\begin{figure}[H]
\centering
\includegraphics[width=.25\textwidth]{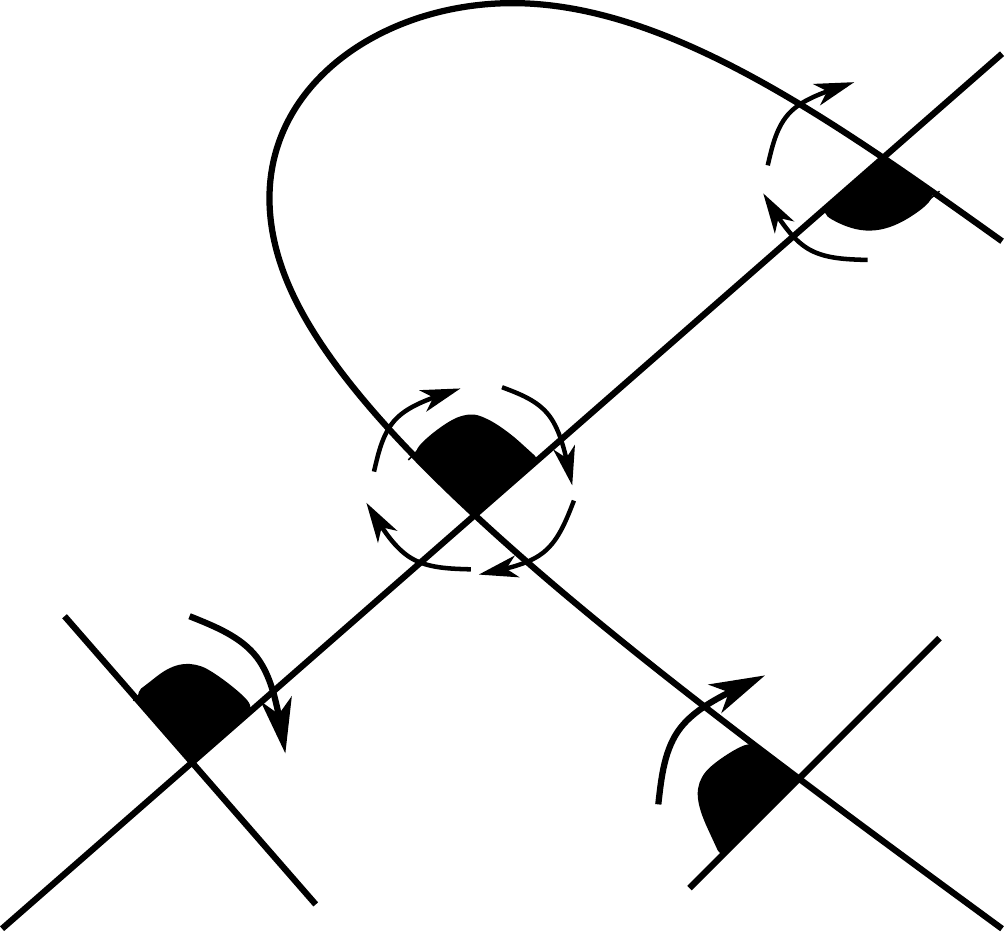}
\caption{}
\label{fig:cases2}
\end{figure}

 The self-interaction site $s$ of the mid-line part may be lying on an involved atom and the removal of $s$ may cause one or both of the segments from the site to belong to uninvolved atoms in $U$.  Figure \ref{fig:last} illustrates a generic picture for this.
 
\begin{figure}[H]
\centering
\includegraphics[width=1\textwidth]{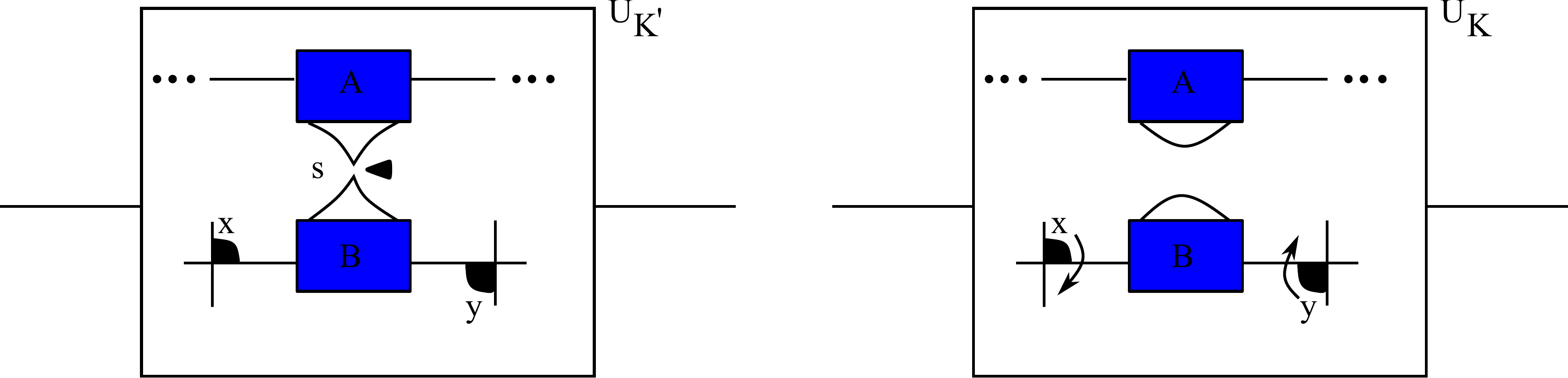}
\caption{}
\label{fig:last}
\end{figure}
 We see in this picture that the site $s$ connects two disjoint parts, denoted by $A$ and $B$, where $A$ and $B$ are riders of the mid-line, included in a bigger rider named as $C$ which we assume to be involved. Here, the atom $B$ is adjacent to two sites in $C$, whose markers get utilized in the transposition. We name these markers as $x$ and $y$.  After the removal of $s$, the atom $C$ is still involved, and thus, the transposition of $x$ and $y$ also takes place in the universe $U_{K}$. On the right-hand side of the figure, it is shown that the atom $B$ gets uninvolved, that is, no marker inside $B$ turns while neighboring markers, for instance $x$ and $y$ turn. Plugging $s$ back forces the marker at $s$ to turn $180$ degrees in the clockwise direction by the argument discussed above.  See also Figure \ref{fig:case5} for an illustration, where $A$ is assumed to be the trivial strand. The same argument also applies for the atom $A$. Therefore the marker at $s$ turns a total of $360$ degrees.

 \begin{figure}[H]
\centering
\includegraphics[width=.4\textwidth]{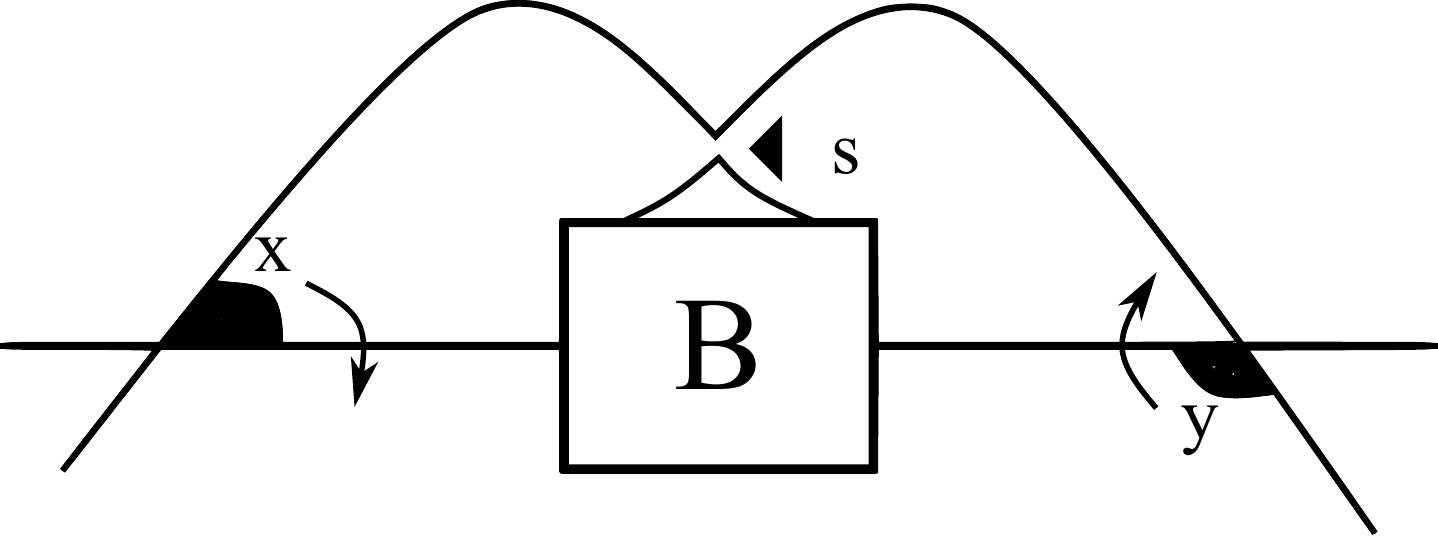}
\caption{}
\label{fig:case5}
\end{figure}

\end{proof}

\

\begin{corollary}

Any state of $U_K$ can be reached from the clocked state by a sequence of clock moves.

\end{corollary}

\begin{definition}\normalfont
 Let $s_1$, $s_2$ be two clock states. We define the \textit{supremum} of $s_1$ and $s_2$ to be the intersection of the clockwise moves applied to obtain $s_1$ and $s_2$ from the clocked state. We define the \textit{infimum} of $s_1$ and $s_2$ to be the union of clockwise moves applied to obtain $s_1$ and $s_2$ from the clocked state. 

We define a partial relation $<$ on the set of all states of a 1-linkoid universe as follows. $s_{2} <  s_{1}$ whenever $s_{2}$ can be reached from $s_{1}$ by a clock move in the clockwise direction.  
\end{definition}

\begin{corollary}
Endowed with the partial relation given above and with the infimum/supremum property,  states of a 1-linkoid diagram form a  lattice.

\end{corollary}

\section{Discussion}

In the case of classical knots and links the clock theorem is essential for proving that the state summation for the Alexander-Conway polynomial gives the same results as Alexander’s original definition using a determinant of a matrix associated with the knot or link diagram. In our generalizations to Mock Alexander Polynomials there are only the state summation definitions and reformulations as permanents. Thus having the clock theorem for knotoids has a different character. It is still the case that the clock theorem provides a structure for the collection of states of the knotoid diagram that is significant for the polynomial invariants that we define from it. The most basic aspect of this relationship is that one can obtain all the states (by marker turnes) from the clockwise or from the counterclockwise state of a given knotoid diagram. Thus the clock theorem provides a specific way to enumerate all the states of the state summation.

The clock states for classical knots and links have been used by Ozsvath and Szabo \cite{Ozsvath} as a basis for the chain complex for Heegard Floer Homology.  In fact purely combinatorial models \cite{Levine, KrizKriz, Ozsvath} for this link homology have been constructed for arbitrary link diagrams using the clock states. It is our intention to generalize these results to knotoids and to use the Clock Theorem for knotoids in this process.

\end{document}